\documentclass[12pt]{amsart}

\usepackage{amscd}
\usepackage{amssymb}
\usepackage{amsmath}
\usepackage{tikz-cd}
\usepackage{url}

\DeclareMathOperator{\Hom}{Hom}
\DeclareMathOperator{\Sh}{Sh}
\DeclareMathOperator{\sign}{sign}
\DeclareMathOperator{\Star}{Star}
\DeclareMathOperator{\bdstar}{\partial \Star}
\DeclareMathOperator{\Irr}{Irr}
\DeclareMathOperator{\Span}{Span}
\DeclareMathOperator{\Supp}{Supp}
\DeclareMathOperator{\im}{im}
\DeclareMathOperator{\coker}{coker}

\DeclareMathOperator{\ch}{ch}
\DeclareMathOperator{\Td}{Td}
\DeclareMathOperator{\id}{Id}

\newcommand\dual[1]{{#1}^{\vee}}
\newcommand{\oo}{\underline{o}}
\newcommand{\EP}{{\textstyle{\bigwedge}}}

\newtheorem{thm}{Theorem}[section]
\newtheorem{cor}[thm]{Corollary}
\newtheorem{prop}[thm]{Proposition}
\newtheorem{lemma}[thm]{Lemma}

\newtheorem{definition}[thm]{Definition}
\newtheorem{example}[thm]{Example}

\theoremstyle{remark}
\newtheorem{remark}[thm]{Remark}

\numberwithin{equation}{subsection}

\begin{document}
\title{On the signature of unimodular fans}
\author{Paul Bressler}
\address{Departamento de Matem\'aticas, Universidad de los Andes, Bogot\'a, Colombia}
\email{p.bressler@uniandes.edu.co}
\author{Diego A. Robayo Bargans}
\address{FB Mathematik, TU Kaiserslautern, 67663 Kaiserslautern, Germany}
\email{robayo@mathematik.uni-kl.de}

\begin{abstract}
N.C.~Leung and V.~Reiner showed that certain convexity conditions on a complete rational simplicial fan determine the sign of the signature of the Poincar\'e pairing on the cohomology of the associated toric variety. The purpose of the present article is to give an ``elementary" proof of their result.
\end{abstract}



\maketitle

\section{Introduction}
N.C.~Leung and V.~Reiner showed in \cite{LR} that certain convexity conditions on a complete rational simplicial fan determine the sign of the signature of the Poincar\'e pairing on the cohomology of the associated toric variety. The purpose of the present article is to give an ``elementary" proof of the following result of \cite{LR}.

Suppose that $\Phi$ is a complete unimodular\footnote{in particular, rational} fan of dimension $2n$ and let $X_\Phi$ denote the corresponding smooth toric variety. Let $\sign(X_\Phi)$ denote the signature of the Poincar\'e pairing on $H(X_\Phi;\mathbb{R})$.

\begin{thm}\label{thm: leung-reiner}
If the fan $\Phi$ is locally convex\footnote{i.e. the support of the star of every ray is convex}, then $(-1)^n\sign(X_\Phi) \geq 0$.
\end{thm}

An outline of the argument in \cite{LR} is as follows. By Hirzebruch's Signature Theorem the signature $\sign(X)$ of an oriented $4k$-dimensional manifold $X$ is given by
\[
\sign(X) = \int_X L(X) ,
\]
where $L(X)$ is the Hirzebruch's $L$-class. If $X$ is a complex manifold of dimension $\dim_\mathbb{C} X = 2n$ the $L$ class may be expressed in terms of Chern classes:
\[
L(X) = \sum_{i=0}^{2n} \ch(\Omega^i_X)\Td(X) .
\]
The Riemann-Roch Theorem implies that
\[
\sign(X) = \sum_{i=0}^{2n} \chi(X;\Omega^i_X) .
\]
The relevant term of the $L$-class of the toric variety $X_\Phi$ takes the form
\[
(-1)^n L(X_\Phi) = \sum_{k=1}^n \sum_{\substack{m_1+\dots+m_k = n,\\ m_i>0,\\
    \rho_1,\dots,\rho_k\in\Phi(1),\\
    \rho_{i}\neq \rho_j}} \left(\text{positive constant}\right)  (-1)^k D_{\rho_1}^{2m_1}\cdot\hdots\cdot D_{\rho_k}^{2m_k} ,
\]
where $D_{\rho_i}$ denotes the divisor corresponding to the ray $\rho_i$. It is therefore sufficient to show that $(-1)^k D_{\rho_1}^{2m_1}\cdot\hdots\cdot D_{\rho_k}^{2m_k} \geqslant 0$ when $\Phi$ is locally convex. If the latter condition is fulfilled, $(-1)^k D_{\rho_1}^{2m_1}\cdot\hdots\cdot D_{\rho_k}^{2m_k}$ is equal to the intersection number of a collection of ample toric divisors on the toric variety $\bigcap\limits_i D_{\rho_i}$ and therefore is nonnegative. 

We adapt the strategy outlined above to the setting of fans. Suppose that $\Phi$ is a complete simplicial fan. Then, the cohomology ring $H(\Phi)$ is defined and carries the Poincar\'e pairing whose signature is denoted $\sign(\Phi)$. Moreover, if the fan $\Phi$ is rational, then $H(\Phi) \cong H(X_\Phi;\mathbb{R})$ and $\sign(\Phi) = \sign(X_\Phi)$.

We assume for the rest of the introduction that the fan $\Phi$ has even dimension $\dim\Phi = 2n$. The fan $\Phi$, viewed as a topological space\footnote{with open sets the subfans of $\Phi$}, carries the sheaves $\Omega^i_\Phi$ of finite dimensional vector spaces which satisfy $H^j(\Phi;\Omega^i_\Phi) = 0$ for $i\neq j$ and $H^i(\Phi;\Omega^i_\Phi) = H^{2i}(\Phi)$. Hence, $\sum\limits_i \chi(\Phi;\Omega^i_\Phi) = \sum\limits_i (-1)^i\dim H^{2i}(\Phi)$. In absence of the signature theorem we show directly (cf. Theorem \ref{thm: signature is euler characteristic}) that
\begin{equation}\label{intro: signature is euler} 
\sign(\Phi) = \sum\limits_i (-1)^i\dim H^{2i}(\Phi) = \sum\limits_i \chi(\Phi;\Omega^i_\Phi).
\end{equation}
All the conclusions drawn so far are purely of combinatorial nature. In order to express $\sign(\Phi)$ in terms of intersection theory we assume from this point on that the fan $\Phi$ is unimodular and, in particular, rational.

For a complete unimodular fan $\Phi$ we construct the Chern character $\ch \colon K(\Phi) \to H(\Phi)$, where $K(\Phi)$ denotes the Grothendieck group of sheaves of finite dimensional vector spaces on $\Phi$, so that $\ch(\Omega^i_\Phi)$ coincides with $\ch( \Omega^i_{X_\Phi})$ under the isomorphism $H(\Phi) \cong H(X_\Phi;\mathbb{R})$. The Todd class $\Td(\Phi)$ is defined so that it corresponds to $\Td(X_\Phi)$. These constructions give rise to a Riemann-Roch type theorem (Theorem \ref{thm: RR}) for sheaves of finite dimensional vector spaces and, hence, to the signature theorem (Theorem \ref{thm: signature}).

If $\Phi$ is locally convex, the inequality $(-1)^n\sign(\Phi) \geqslant 0$ (Theorem \ref{thm: sign of signature}) is reduced, using the formula for the $L$-class and a result of M.~Brion, to the positivity of mixed volumes.

A synopsis of the paper is as follows. In Section \ref{section: fans} we recall the basic facts regarding fans. In Section \ref{section: sheaves} we review the theory of sheaves on fans and introduce the relevant examples of such. In Section \ref{section: the grothendieck group} we describe the structure of the Grothendieck group of sheaves of finite dimensional vector spaces on a fan. Section \ref{section: cohomology} is devoted to cohomology of simplicial fans, including the Poincar\'e pairing and the formalism of Gysin maps. In Section \ref{section: signature} we prove the equality \eqref{intro: signature is euler}. In Section \ref{section: riemann-roch} we construct the Chern character on the Grothendieck group of sheaves of finite dimensional vector spaces on a complete unimodular fan and prove a Riemann-Roch type theorem, obtaining Hirzebruch's signature theorem as a corollary. In Section \ref{section: locally convex fans} we apply the above results to locally convex fans and provide a proof of Theorem \ref{thm: leung-reiner}.

\section{Fans}\label{section: fans}

\subsection{Cones}
Suppose that $V$ is a finite dimensional real vector space. In what follows, by ``a cone in $V$'' we shall mean ``a closed convex polyhedral cone with vertex at the origin''. The origin of $V$ is a cone and will be denoted $\oo$. Given two cones $\tau$ and $\sigma$ in $V$ we shall write $\tau\leqslant\sigma$ whenever $\tau$ is a face of $\sigma$. This defines a partial order on the set of cones with the unique minimal element being the origin $\oo$.

For a cone $\sigma$ we denote the dimension of $\sigma$ by $d(\sigma)$. Note that $d(\sigma) = 0$ if and only if $\sigma = \oo$. The assignment $\sigma\mapsto d(\sigma)$ defines a grading on the partially ordered set of cones in $V$.

A one-dimensional cone is called a \emph{ray}. A codimension one face of a cone is called a \emph{facet}.

A cone of dimension $k$ is \emph{simplicial} if it has $k$ one-dimensional faces.

\subsection{Fans}
A \emph{fan $\Phi$ in $V$} is a finite collection of cones in $V$ satisfying
\begin{enumerate}
\item any two cones in $\Phi$ intersect along a common face;
\item if $\sigma\in\Phi$ and $\tau\leqslant\sigma$, then $\tau\in\Phi$.
\end{enumerate}

The \emph{support} of a fan $\Phi$, denoted $|\Phi| \subset V$, is the union of all cones of $\Phi$.

A fan $\Phi$ in $V$ is called \emph{complete} if the union of all cones of $\Phi$ is equal to $V$.

For a fan $\Phi$ and $i\in\mathbb{Z}$ let
\[
\Phi(i) := \{ \sigma\in\Phi \mid d(\sigma) = i \}
\]
For a subset $S \subset \Phi$ we denote by $[S]$ the subfan generated by $S$. In particular, for $\sigma\in\Phi$ let $[\sigma] := [\{\sigma\}]$; let $\partial\sigma := [\sigma]\smallsetminus\{\sigma\}$, the subfan of $[\sigma]$ generated by proper faces of $\sigma$.

\subsection{Fans as topological spaces}\label{subsection: Fans as topological spaces}
Suppose that $\Phi$ is a fan. We shall consider the partially ordered set $\Phi$ as a topological space with open subset the subfans of $\Phi$.

The irreducible\footnote{An open subset is irreducible if it is not a union of non-empty open subsets properly contained in it.} open subsets of $\Phi$ are the subfans $[\sigma]$, $\sigma\in\Phi$. Let $\Irr(\Phi)$ denote the partially ordered (by inclusion) set of irreducible open subsets of $\Phi$. Note that the assignment $\sigma\mapsto [\sigma]$ defines an isomorphism $\Phi \to \Irr(\Phi)$ of partially ordered sets.

For $\tau\in\Phi$ let
\begin{eqnarray*}
\Star(\tau) = \Star_\Phi(\tau) & := & \{ \sigma\in\Phi \mid \sigma\geqslant\tau \} , \\
\bdstar(\tau) = \bdstar_\Phi(\tau) & := & [\Star(\tau)] \smallsetminus \Star(\tau).
\end{eqnarray*}
The subset $\Star(\tau)$ is the closure of $\{\tau\}$, hence a closed subset. Its image under the projection $V \to V/\Span(\tau)$ is a fan denoted $\overline{\Star}(\tau)$. For $\sigma\in\Star(\tau)$ we denote the image of $\sigma$ by $\sigma/\tau\in\overline{\Star}(\tau)$.

\subsection{Rational fans}
Suppose that $\Lambda$ is a finitely generated free abelian group. Let $V = \Lambda\otimes_\mathbb{Z}\mathbb{R}$. Thus, $\Lambda$ is a lattice in $V$. A fan $\Phi$ in $V$ is \emph{rational} if every ray $\rho\in\Phi(1)$ contains a non-zero element of $\Lambda$.

For a ray $\rho\in\Phi(1)$ we denote by $v_\rho$ the primitive vector, i.e. the generator of the monoid $\rho\cap\Lambda$.

\subsection{Unimodular fans}
Suppose that $\Phi$ is a rational simplicial fan in $V = \Lambda\otimes_\mathbb{Z}\mathbb{R}$. The fan $\Phi$ is \emph{unimodular} if for any cone $\sigma\in\Phi$, the set $\{v_\rho \mid \rho\in [\sigma](1) \}$ is a part of a basis of $\Lambda$.

\begin{example}\label{example: Pn}
In what follows we denote by $\mathbb{P}^n$ the complete simplicial fan in $\mathbb{R}^n$ whose rays are generated by the standard basis vectors $e_i$, $i = 1,\ldots,n$ and the vector $-\sum\limits_{i=1}^n e_i$. Thus, $\mathbb{P}^n$ is rational with respect to the standard lattice $\sum\limits_{i=1}^n \mathbb{Z}e_i \subset \mathbb{R}^n$ and unimodular.
\end{example}

\subsection{Subdivisions}
Suppose that $\Phi$ is a fan in $V$. A fan $\Psi$ is said to be a \emph{subdivision of $\Phi$} if it has the same support as $\Phi$ and every cone of $\Phi$ is a union of cones of $\Psi$.

Suppose that $\Psi$ is a subdivision of $\Phi$. The assignment
\[
\Psi\ni\sigma \mapsto \pi(\sigma) := \text{the smallest cone of $\Phi$ containing $\sigma$}
\]
defines a continuous map $\pi \colon \Psi \to \Phi$.

\subsection{Star subdivisions}
Let $\sigma\in\Phi$. Suppose that $\rho$ is a ray in $V$ which has non-trivial intersection with the relative interior of the cone $\sigma$. Then, the collection of cones
\begin{equation}\label{star subdivision}
\Phi\smallsetminus\Star_\Phi(\sigma) \cup \{\xi+\rho \mid \xi\in[\Star(\sigma)]\textnormal{ and }\sigma\cap \xi=\oo \}
\end{equation}
is a fan. Moreover, the fan \eqref{star subdivision} is a subdivision of the fan $\Phi$ called the \emph{star subdivision of $\Phi$ at $\sigma$ along $\rho$}. If $\Phi$ is unimodular, then the fan \eqref{star subdivision} is called a regular star subdivision if $\rho$ is the ray generated by the sum of primitive vectors of the rays of $\sigma$.

\begin{remark}\label{remark: simplicial star subdivision}
If $\Phi$ is simplicial, then so is any star subdivision of $\Phi$. If $\Phi$ is unimodular, then so is any regular star subdivision.
\end{remark}

\begin{thm}[\cite{W} Corollary 8.3, Theorem 13.3]\label{thm: subdivision theorem}
Any two fans $\Phi$ and $\Psi$ with the same support are related by a sequence of star subdivisions. If, in addition, $\Phi$ and $\Psi$ are unimodular, then they are related by a sequence of regular subdivisions.
\end{thm}

\section{Sheaves hit the fans}\label{section: sheaves}

\subsection{Sheaves of vector spaces}
Let $\Sh(\Phi)$ denote the category of sheaves of real vector spaces on  $\Phi$ considered as a topological space as in \ref{subsection: Fans as topological spaces}. Let $\Sh_{fin}(\Phi)$ denote the full subcategory of sheaves of finite dimensional vector spaces.

For $\sigma\in\Phi$
\begin{itemize}
\item let $i_{[\sigma]} \colon [\sigma] \hookrightarrow \Phi$ denote the open embedding of the of the irreducible open set,
\item let $i_\sigma \colon \{\sigma\} \hookrightarrow \Phi$ denote the inclusion; the embedding $i_\sigma$ is locally closed and closed (respectively, open) if and only if $\sigma$ is maximal (respectively, minimal, i.e. $\sigma = \oo$).
\end{itemize}

\begin{prop}
{~}
\begin{enumerate}
\item For any $\sigma\in\Phi$ and any vector space $W$ the sheaf $i_{[\sigma] !}W = W_{[\sigma]}$ is projective.
\item For any $\sigma\in\Phi$ and any vector space $W$ the sheaf $i_{\sigma *}W = W_{\Star(\sigma)}$ is injective.
\item The categories $\Sh(\Phi)$ and $\Sh_{fin}(\Phi)$ have enough projectives.
\item The categories $\Sh(\Phi)$ and $\Sh_{fin}(\Phi)$ have enough injectives.
\end{enumerate}
\end{prop}
\begin{proof}
{~}
\begin{enumerate}
\item The functor $\Sh(\Phi) \to \mathbb{R}-\mathrm{mod}$ given by $F \mapsto \Hom_{\Sh(\Phi)}(W_{[\sigma]}, F) = \Hom_\mathbb{R}(W,F_\sigma)$ is exact.

\item The functor $\Sh(\Phi)^\mathrm{op} \to \mathbb{R}-\mathrm{mod}$ given by $F \mapsto \Hom_{\Sh(\Phi)}(F, W_{\Star(\sigma)}) = \Hom_\mathbb{R}(F_\sigma, W)$ is exact.

\item For $F\in \Sh(\Phi)$ the canonical map $\bigoplus\limits_{\sigma\in\Phi} (F_\sigma)_{[\sigma]} \to F$ is an epimorphism.

\item For $F\in \Sh(\Phi)$ the canonical map $F \to \prod\limits_{\sigma\in\Phi} i_{\sigma *}F_\sigma$ is a monomorphism.
\end{enumerate}
\end{proof}

For $F\in\Sh(\Phi)$ the \emph{support of $F$}, denoted $\Supp(F)$ is defined by
\[
\Supp(F) = \{ \sigma\in\Phi \mid F_\sigma\neq 0 \} .
\]

Recall that a sheaf $F$ on a space $X$ is \emph{flabby} if for any open subset $U \subset X$ the restriction map $F(X) \to F(U)$ is surjective.

\subsection{The cellular complex}
Suppose that $\Phi$ is a fan in a vector space $V$ of dimension $n := \dim_\mathbb{R} V$. For $F\in \Sh(\Phi)$ let
\[
C^i(\Phi;F) := \bigoplus\limits_{\sigma\in\Phi(n-i)} F_\sigma .
\]
In order to define the differential
\begin{equation}\label{cellular differential}
d^{i} \colon C^i(\Phi;F) \to C^{i+1}(\Phi;F)
\end{equation}
we fix a choice of an orientation for each cone $\sigma\in\Phi$.

For $\sigma\in\Phi$ and a facet $\tau\in [\sigma](d(\sigma)-1)$ the sign $\epsilon_{\sigma\tau}$ is defined by
\[
\epsilon_{\sigma\tau} := \begin{cases} +1 & \text{if orientations of $\sigma$ and $\tau$ agree} \\ -1 & \text{otherwise} \end{cases}
\]
The map \eqref{cellular differential} is defined as the sum of the restriction maps with signs:
\[
d^{i} = \bigoplus\limits_{\sigma\in\Phi(n-i)} \sum\limits_{\tau\in [\sigma](d(\sigma)-1)} \epsilon_{\sigma\tau}\cdot (F_\sigma \to F_\tau) .
\]
The assignment $F \mapsto C(\Phi;F)$ defines an exact functor on $\Sh(\Phi)$ with values in the category of complexes of vector spaces.

\begin{prop}[\cite{ICNP}, Proposition 3.5]
Suppose that $\Phi$ is a complete fan. Then, $C(\Phi; \bullet)$ and $\mathbf{R}\Gamma(\Phi;\bullet)$ are canonically isomorphic as functors $D^b(\Sh(\Phi)) \to D^b(\mathbb{R}-\mathrm{mod})$.
\end{prop}

\subsection{Examples of sheaves}\label{subsection: Examples of sheaves}
Suppose that $\Phi$ is fan in $V$.

For $\sigma\in\Phi$ let
\begin{itemize}
\item $\dual{\sigma} := \dual{\Span(\sigma)}$
\item $\sigma^\perp := \{f\in\dual{V} \mid f\vert_\sigma = 0 \} \subset \dual{V}$ .
\end{itemize}
Note that if $\tau\leqslant\sigma$, then $\sigma^\perp\subseteq\tau^\perp$.

Let $\Omega^1_\Phi$ denote the subsheaf of $\dual{V}_\Phi$ defined by assignment
\[
\Omega^1_\Phi\colon\sigma \mapsto \Omega^1_{\Phi,\sigma} = \sigma^\perp .
\]
Let $\Omega^0_\Phi = \mathbb{R}_\Phi$ and let $\Omega^q_\Phi = \bigwedge^q\Omega^1_\Phi$ for $q=1,2,\ldots$.

The sheaf $\mathcal{G}$ is defined by the short exact sequence
\[
0 \to \Omega^1_\Phi \to \dual{V}_\Phi \to \mathcal{G} \to 0 .
\]
Since $\mathcal{G}_\sigma \cong \dual{\sigma}$, sections of $\mathcal{G}$ are continuous cone-wise linear functions. The restriction maps $\dual{\sigma} \to \dual{\rho}$, $\rho\in[\sigma](1)$, give rise to the morphism of sheaves
\begin{equation}\label{cone-wise linear functions to rays}
\mathcal{G} \to \bigoplus\limits_{\rho\in\Phi(1)} i_{\rho *}\dual{\rho} = \bigoplus\limits_{\rho\in\Phi(1)} \dual{\rho}_{\Star(\rho)}.
\end{equation}

\begin{lemma}\label{lemma: cone-wise linear functions to rays}
The map \eqref{cone-wise linear functions to rays} is a monomorphism and an isomorphism if an only if $\Phi$ is simplicial.
\end{lemma}
\begin{proof}
Let $\sigma\in\Phi$. Since the rays of $\sigma$ generate $\Span(\sigma)$ it follows that the map $\bigoplus\limits_{\rho\in[\sigma](1)} \Span(\rho) \to \Span(\sigma)$ is surjective. Hence, the dual map $\mathcal{G}_\sigma \cong \dual{\sigma} \to \bigoplus\limits_{\rho\in [\sigma](1)} \dual{\rho}$ is injective.
If $\sigma$ is simplicial, then the cardinality of $[\sigma](1)$ coincides with the dimension of $\sigma^\vee$, hence the map $\dual{\sigma} \to \bigoplus\limits_{\rho\in [\sigma](1)} \dual{\rho}$ is an isomorphism.
\end{proof}

For $\rho\in\Phi(1)$ let
\begin{equation}\label{line bundles associated to rays}
\mathcal{O}(\rho) := \ker(\dual{\rho}_\Phi \to i_{\rho *}\dual{\rho}) .
\end{equation}
For $\sigma\in\Phi$ let
\begin{equation}\label{line bundles associated to cones}
\mathcal{O}(\sigma) := \bigotimes\limits_{\rho\in[\sigma](1)} \mathcal{O}(\rho) .
\end{equation}
By convention we set $\mathcal{O} = \mathcal{O}(\oo) = \mathbb{R}_\Phi$.

\begin{remark}
The sheaves $\mathcal{O}(\rho)$, $\rho\in\Phi(1)$, play a key role in the construction of the Chern character map (cf. \ref{subsection: chern character}). For $\Phi$ rational they are analogs of the line bundles $\mathcal{O}(-D_\rho)$, where $D_\rho$ is the irreducible divisor which corresponds to the ray $\rho$, on the toric variety $X_\Phi$ .
\end{remark}

\begin{lemma}\label{lemma: small to big}
The diagram
\[
\begin{CD}
\dual{V}_\Phi @>>> \mathcal{G} \\
@VVV @VV{\eqref{cone-wise linear functions to rays}}V \\
\bigoplus\limits_{\rho\in\Phi(1)} \dual{\rho}_\Phi @>>> \bigoplus\limits_{\rho\in\Phi(1)} i_{\rho *}\dual{\rho}
\end{CD}
\]
is commutative.
\end{lemma}
\begin{proof}
For $\sigma\in\Phi$ the corresponding diagram
\[
\begin{CD}
\dual{V} @>>> \dual{\sigma} \\
@VVV @VVV \\
\bigoplus\limits_{\rho\in\Phi(1)} \dual{\rho} @>\textnormal{pr}>> \bigoplus\limits_{\rho\in[\sigma](1)}\dual{\rho}
\end{CD}
\]
of maps of stalks at $\sigma$ is commutative where 
\[\textnormal{pr}:\bigoplus\limits_{\rho\in\Phi(1)} \dual{\rho}\to \bigoplus\limits_{\rho\in[\sigma](1)}\dual{\rho}\]
is the natural projection.
\end{proof}

It follows from Lemma \ref{lemma: small to big} that the composition
\[
\Omega^1_\Phi = \ker(\dual{V}_\Phi \to \mathcal{G}) \to \dual{V}_\Phi \to \bigoplus\limits_{\rho\in\Phi(1)} \dual{\rho}_\Phi
\]
factors canonically through $\bigoplus\limits_{\rho\in\Phi(1)}\mathcal{O}(\rho) = \ker\left(\bigoplus\limits_{\rho\in\Phi(1)} \dual{\rho}_\Phi \to \bigoplus\limits_{\rho\in\Phi(1)} i_{\rho *}\dual{\rho}\right)$ so that the diagram
\begin{equation}\label{map of ses}
\begin{CD}
0 @>>> \Omega^1_\Phi @>>> \dual{V}_\Phi @>>> \mathcal{G} @>>> 0 \\
& & @VVV @VVV @VV{\eqref{cone-wise linear functions to rays}}V \\
0 @>>> \bigoplus\limits_{\rho\in\Phi(1)}\mathcal{O}(\rho) @>>> \bigoplus\limits_{\rho\in\Phi(1)} \dual{\rho}_\Phi @>>> \bigoplus\limits_{\rho\in\Phi(1)} i_{\rho *}\dual{\rho}  @>>> 0
\end{CD}
\end{equation}
with exact rows is commutative, giving rise to the complex
\begin{equation}\label{complex forms big}
\Omega^1_\Phi \to \bigoplus\limits_{\rho\in\Phi(1)}\mathcal{O}(\rho) \to \coker\left(\dual{V}_\Phi \to \bigoplus\limits_{\rho\in\Phi(1)} \dual{\rho}_\Phi\right) .
\end{equation}

\begin{prop}
Suppose that $\Phi$ is a complete simplicial fan. Then,
\begin{enumerate}
\item $\coker\left(\dual{V} \to \bigoplus\limits_{\rho\in\Phi(1)} \dual{\rho}\right) \cong H^1(\Phi;\Omega^1_\Phi)$

\item The sequence
\begin{equation}\label{ses forms big}
0 \to \Omega^1_\Phi \to \bigoplus\limits_{\rho\in\Phi(1)}\mathcal{O}(\rho) \to H^1(\Phi;\Omega^1_\Phi)_\Phi \to 0
\end{equation}
deduced from \eqref{complex forms big} is exact.
\end{enumerate}
\end{prop}
\begin{proof}
{~}
\begin{enumerate}
\item The long exact sequence in cohomology associated to the short exact sequence of sheaves
\[
0 \to \Omega^1_\Phi \to \dual{V}_\Phi \to \mathcal{G} \to 0 .
\]
reduces to the short exact sequence
\[
0 \to \dual{V} \to \bigoplus\limits_{\rho\in\Phi(1)}\dual{\rho} \to H^1(\Phi;\Omega^1_\Phi) \to 0
\]
using
\begin{itemize}
\item $H^0(\Phi;\dual{V}_\Phi) = \dual{V}_\Phi$ and $H^i(\Phi;\dual{V}_\Phi) = 0$ for $i\neq 0$ since $\dual{V}_\Phi$ is injective;
\item $\mathcal{G} \cong \bigoplus\limits_{\rho\in\Phi(1)} i_{\rho *}\dual{\rho}$, hence $H^0(\Phi;\mathcal{G}) = \bigoplus\limits_{\rho\in\Phi(1)}\dual{\rho}$ and $H^i(\Phi;\mathcal{G}) = 0$ for $i\neq 0$ since $\mathcal{G}$ is injective;
\item $H^0(\Phi;\Omega^1_\Phi) = 0$ since $\Omega^1_\Phi$ is not supported on top-dimensional cones.
\end{itemize}

\item Under the present assumptions 
\begin{itemize}
\item the vertical maps in \eqref{map of ses} are monomorphisms,
\item the map \eqref{cone-wise linear functions to rays} is an isomorphism.
\end{itemize}
The snake lemma implies that the map
\[
\coker\left(\Omega^1_\Phi \to \bigoplus\limits_{\rho\in\Phi(1)}\mathcal{O}(\rho)\right) \to \coker\left(\dual{V} \to \bigoplus\limits_{\rho\in\Phi(1)} \dual{\rho}\right) \cong H^1(\Phi;\Omega^1_\Phi)
\]
is an isomorphism.
\end{enumerate}
\end{proof}

\subsection{The structure sheaf}\label{subsection: the structure sheaf}
Let $\mathcal{A}_\Phi := \mathrm{S}(\mathcal{G}(-2))$ denote the symmetric algebra on $\mathcal{G}(-2)$. The sheaf $\mathcal{A}_\Phi$ is a sheaf of graded algebras of continuous cone-wise polynomial functions, graded so that cone-wise linear functions have degree two.

\begin{lemma}[\cite{ICNP}, Lemma 4.6]
The sheaf $\mathcal{A}_\Phi$ is flabby if an only if $\Phi$ is simplicial.
\end{lemma}

Suppose that $\pi\colon \Psi \to \Phi$ is a subdivision. Restriction of cone-wise polynomial functions gives rise to the morphism of sheaves of graded algebras $\pi^* \colon \mathcal{A}_\Phi \to \pi_*\mathcal{A}_\Psi$. In other words, subdivision is a morphism of ringed spaces $\pi\colon (\Psi,\mathcal{A}_\Psi) \to (\Phi,\mathcal{A}_\Phi)$.

\subsection{Quasi-convex fans}\label{subsection: quasi-conves fans}
According to \cite{BBFK}, a simplicial fan $\Phi$ in $V$ is \emph{quasi-convex} if $\mathcal{A}_\Phi(\Phi)$ is a free module over $A_V := \mathrm{S}(\dual{V}(-2))$. By Theorem 4.4 of loc. cit. a purely $n$-dimensional fan $\Phi$ is quasi-convex if and only if the support of its boundary fan is a real homology manifold.

It follows that
\begin{itemize}
\item any complete fan is quasi-convex;
\item if $\Phi$ is a complete simplicial fan and $\tau\in\Phi$, then $[\Star_\Phi(\tau)]$ is quasi-convex
\item and so is $\Phi\setminus\Star_\Phi(\tau)$ (since it has the same boundary as $[\Star_\Phi(\tau)]$).
\end{itemize}

\section{The Grothendieck group of sheaves on a fan}\label{section: the grothendieck group}
Suppose that $\Phi$ is a \emph{simplicial} fan in $V$; let $n :=\dim V$.

\subsection{$\mathbf{K(\Phi)}$}
We denote by $K(\Phi)$ the Grothendieck group of the category $\Sh_{fin}(\Phi)$ of sheaves of finite dimensional vector spaces on $\Phi$. For $F\in\Sh_{fin}(\Phi)$ we denote the corresponding element of $K(\Phi)$ by $[F]$.

Since the tensor product of sheaves of vector spaces is exact, it induces a binary operation $(\cdot)\bullet(\cdot)$ on $K(\Phi)$ such that $[F]\bullet[G] = [F\otimes G]$ for $F,G\in\Sh_{fin}(\Phi)$, which is associative and commutative with unit $1 := [\mathbb{R}_\Phi]$.

Let
\[
\chi\colon K(\Phi) \to \mathbb{Z}
\]
denote the map defined by
\[
[F] \mapsto \chi([F]) :=   \chi(\Phi; F) =\sum\limits_i (-1)^i\dim C^i(\Phi;F)
\]
for $F\in\Sh_{fin}(\Phi)$.

\begin{remark}
If $\Phi$ is complete, then $\chi(\Phi; F) =\sum\limits_i (-1)^i\dim H^i(\Phi;F)$.
\end{remark}
\begin{lemma}\label{lemma: line bundles vs injectives}
For $\sigma\in\Phi$
\begin{align*} 
[\mathbb{R}_{\Star(\sigma)}] & = \sum_{\tau\leqslant \sigma}(-1)^{d(\tau)}[\mathcal{O}(\tau)], \\
[\mathcal{O}(\sigma)] & = \sum_{\tau\leqslant \sigma}(-1)^{d(\tau)}[\mathbb{R}_{\Star(\tau)}].
\end{align*}
\end{lemma}
\begin{proof}
For $\rho\in\Phi(1)$ the definition of $\mathcal{O}(\rho)$ implies that
\[
[\mathcal{O}(\rho)] = 1-[\mathbb{R}_{\Star(\rho)}],\ \ \ [\mathbb{R}_{\Star(\rho)}] = 1-[\mathcal{O}(\rho)] .
\]
Since $\Phi$ is simplicial, for $\sigma\in\Phi$
\[
[\mathbb{R}_{\Star(\sigma)}] = \prod\limits_{\rho\in[\sigma](1)}[\mathbb{R}_{\Star(\rho)}] = \prod\limits_{\rho\in[\sigma](1)} (1-[\mathcal{O}(\rho)]) = \sum_{\tau\leqslant \sigma}(-1)^{d(\tau)}[\mathcal{O}(\tau)]
\]
and
\[
[\mathcal{O}(\sigma)] = \prod\limits_{\rho\in[\sigma](1)}(1-[\mathbb{R}_{\Star(\rho)}]) = \sum_{\tau\leqslant \sigma}(-1)^{d(\tau)}[\mathbb{R}_{\Star(\tau)}].
\]
\end{proof}
For $\sigma\in\Phi$ let $\mathbb{R}_{\underline{\sigma}}$ denote the sheaf defined by
\[
(\mathbb{R}_{\underline{\sigma}})_\tau = \begin{cases} \mathbb{R} & \text{if $\sigma = \tau$,} \\
0 & \text{otherwise}.
\end{cases}
\]
\begin{prop}\label{prop: bases for K}
{~}
The group $K(\Phi)$ is free of finite rank equal to the cardinality of $\Phi$. Moreover:
\begin{enumerate}
\item The collection of classes $\{\mathbb{R}_{\underline{\sigma}}\}$, $\sigma\in\Phi$ is a basis for $K(\Phi)$.
\item The collection of classes $[\mathbb{R}_{\Star(\sigma)}]$, $\sigma\in\Phi$ is a basis for $K(\Phi)$.
\item The collection of classes $[\mathcal{O}]$, $[\mathcal{O}(\sigma)]$, $\oo\neq\sigma\in\Phi$ is a basis for $K(\Phi)$.
\end{enumerate}
\end{prop}
\begin{proof}
Let $\mathcal{Q}(\Phi)$ denote the quiver given by dual poset of $\Phi$. This quiver naturally has no oriented cycles and its vertices correspond to cones of $\Phi$. In particular, they also correspond to the irreducible open sets of $\Phi$ and therefore a sheaf $F\in \Sh_{fin}(\Phi)$ gives rise to a representation of $\mathcal{Q}(\Phi)$. This correspondence is an equivalence of categories, and hence induces an isomorphism between the corresponding Grothendieck groups. Under this correspondence the sheaves $\{\mathbb{R}_{\underline{\sigma}}\}_{\sigma\in\Phi}$ correspond to a full list of simple representations of $\mathcal{Q}(\Phi)$ (Theorem 1.10 of \cite{K}) and they form a basis of $K(\mathcal{Q}(\Phi))$ (Theorem 1.15 loc. cit.). Hence, the sheaves $\{\mathbb{R}_{\underline{\sigma}}\}_{\sigma\in\Phi}$ form a basis of $K(\Phi)$ and the latter group is free of rank equal to the cardinality of $\Phi$.
The sheaves $[\mathbb{R}_{\Star(\sigma)}]$, $\sigma\in\Phi$ are the indecomposable injective objects in $\Sh_{fin}(\Phi)$. Since the category $\Sh_{fin}(\Phi)$ has enough injective objects, it follows that the collection $[\mathbb{R}_{\Star(\sigma)}]$, $\sigma\in\Phi$, generates $K(\Phi)$. As this collection has the size of a basis, it follows that it must be a basis. Since the collection $[\mathbb{R}_{\Star(\sigma)}]$, $\sigma\in\Phi$ is a basis, Lemma \ref{lemma: line bundles vs injectives} shows that the collection $[\mathcal{O}] = [\mathbb{R}_{\Star(\oo)}]$, $[\mathcal{O}(\sigma)]$, $\oo\neq\sigma\in\Phi$ is a basis as well.
\end{proof}

\subsection{The product formula}
For $S\subseteq\Phi(1)$ let
\[
\langle S\rangle := \{\sigma\in\Phi \mid [\sigma](1) \subseteq S \}.
\]
The subset $\langle S \rangle \subset \Phi$ is the largest subfan of $\Phi$ such that every $\sigma\in \langle S\rangle$ is generated by the rays in $S$.

\begin{prop}\label{prop: product formula}
For any $S\subset \Phi(1)$
\begin{equation}\label{product formula}
\prod_{\rho\in S}[\mathcal{O}(\rho)] = \sum_{\tau\in \langle S\rangle} (-1)^{\dim\overline{\Star}_{\langle S\rangle}(\tau)} \chi\left(\overline{\Star}_{\langle S\rangle}(\tau);\mathbb{R} \right)\cdot[\mathcal{O}(\tau)].
\end{equation}
\end{prop}
\begin{proof}
By the definition of $\langle S \rangle$,
\begin{align*}
\prod_{\rho\in S}[\mathcal{O}(\rho)] & = \prod_{\rho\in S}(1-[\mathbb{R}_{\Star(\rho)}]) \\
& = \sum_{\sigma\in\langle S \rangle} (-1)^{d(\sigma)}[\mathbb{R}_{\Star(\sigma)}] \\
\text{by Lemma \ref{lemma: line bundles vs injectives}}\ \ \ & = \sum_{\sigma\in\langle S \rangle}\sum_{\tau\leqslant \sigma}(-1)^{d(\sigma) - d(\tau)}[\mathcal{O}(\tau)] \\
& = \sum_{\tau\in\langle S \rangle}\sum_{\sigma\in\langle S \rangle\cap\Star(\tau)}(-1)^{d(\sigma) - d(\tau)}[\mathcal{O}(\tau)] \\
& = \sum_{\tau\in\langle S \rangle}\sum_{\sigma/\tau\in\overline{\Star}_{\langle S \rangle}(\tau)}(-1)^{d(\sigma/\tau)}[\mathcal{O}(\tau)] \\
& = \sum_{\tau\in\langle S \rangle}(-1)^{\dim\overline{\Star}_{\langle S\rangle}(\tau)}\chi\left(\overline{\Star}_{\langle S\rangle}(\tau);\mathbb{R} \right)[\mathcal{O}(\tau)] .
\end{align*}
\end{proof}

\begin{example}\label{example: projective space}
Consider the example $\Phi = \mathbb{P}^n$ (see Example \ref{example: Pn}), $S = \mathbb{P}^n(1)$. Then, $\langle\mathbb{P}^n(1)\rangle = \mathbb{P}^n$. For $\tau\in\mathbb{P}^n(k)$ the fan $\overline{\Star}_{\mathbb{P}^n}(\tau)$ is isomorphic to $\mathbb{P}^{n-k}$ with $\chi\left(\mathbb{P}^{n-k};\mathbb{R} \right) = 1$.

The product formula \eqref{product formula} says
\[
\prod_{\rho\in \mathbb{P}^n(1)}[\mathcal{O}(\rho)] = \sum_{\tau\in \mathbb{P}^n} (-1)^{n-d(\tau)}[\mathcal{O}(\tau)].
\]
\end{example}

\section{(Intersection) cohomology of simplicial fans}\label{section: cohomology}
Throughout this section $\Phi$ is a \emph{complete} simplicial fan in $V$ of dimension $\dim\Phi = \dim V = n$. Let $A = A_V := \operatorname{S}(\dual{V}(-2))$ be the algebra of regular functions on $V$ graded so that linear functions have degree 2. Let $A^+\subset A$ denote the ideal of elements of positive degree.  For a graded $A$-module $M$ we denote by $\overline{M}$ the graded vector space $M/A^+M$.

\subsection{Intersection cohomology}
Recall the ``structure sheaf'' of graded algebras $\mathcal{A}_\Phi$ introduced in \ref{subsection: the structure sheaf}.

Since $\Phi$ is a complete simplicial fan, it follows that (see \cite{ICNP}, Theorem 4.7),
\begin{itemize}
\item $\mathcal{A}_\Phi$ is flabby, hence $H^i(\Phi;\mathcal{A}_\Phi) = 0$ for $i\neq 0$;
\item the $A$-module $H^0(\Phi;\mathcal{A}_\Phi)$ is free
\end{itemize}
The \emph{(intersection) cohomology of $\Phi$} is defined as the graded vector space
\[
H(\Phi) := \overline{H^0(\Phi;\mathcal{A}_\Phi)} .
\]

\begin{prop}\label{prop: coh of forms}
{\ }
\begin{enumerate}
\item $H^i(\Phi;\Omega^j_\Phi) = 0$ for $i\neq j$;
\item $H(\Phi) \cong \bigoplus\limits_{i} H^i(\Phi;\Omega^i_\Phi)(-2i)$
\end{enumerate}
\end{prop}
\begin{proof}
Let
\[
K^j := A_\Phi\otimes\Omega_\Phi^{-j}(-2j).
\]
Let $\partial\colon K^j \to K^{j+1}$ denote the map given by
\[
f\otimes\xi_1\wedge\cdots\xi_j\mapsto\sum_{k=1}^j
(-1)^{k-1}f\xi_k\otimes\xi_1\cdots\widehat\xi_k\cdots\xi_j .
\]
Then, $\partial\circ\partial = 0$, i.e. $(K^\bullet,\partial)$ is a complex. The surjection $A_\Phi \to \mathcal{A}_\Phi$ extends to the map of complexes $K^\bullet \to \mathcal{A}_\Phi$ which is a quasi-isomorphism. Since the complex $K^\bullet$ is a resolution of $\mathcal{A}_\Phi$ by flat $A_\Phi$-modules, it follows that
\[
\mathcal{A}_\Phi\otimes^\mathbb{L}_{A_\Phi}\mathbb{R}_\Phi\cong K^\bullet\otimes_{A_\Phi}\mathbb{R}_\Phi =\bigoplus_j\Omega^j_\Phi(-2j)[j] \ .
\]

Since $\mathcal{A}_\Phi$ is flabby and $H^0(\Phi;\mathcal{A}_\Phi)$ is free over $A_V$ it follows that
\[
\mathbf{R}\Gamma(\Phi;\mathcal{A}_\Phi\otimes^\mathbb{L}_{A_\Phi}\mathbb{R}_\Phi)\cong
\mathbf{R}\Gamma(\Phi;\mathcal{A}_\Phi)\otimes^\mathbb{L}_A\mathbb{R}\cong
H^0(\Phi;\mathcal{A}_\Phi)\otimes_A\mathbb{R} = H(\Phi) \ .
\]
Therefore, $H^i(\Phi;\bigoplus\limits_j\Omega^j_\Phi(-2j)[j]) = 0$ for $i\neq 0$, i.e. $H^i(\Phi;\Omega^j_\Phi) = 0$ for $i\neq j$, and $H(\Phi) \cong \bigoplus\limits_{i} H^i(\Phi;\Omega^i_\Phi)(-2i)$.
\end{proof}

\subsection{Brion's functional}\label{subsection: brion's functional}
We recall the construction of the trace map on the cohomology of a complete fan due to M.~Brion (\cite{B}).
In what follows we fix a volume form $\Omega_V$ on $V$.

\begin{thm}[M.~Brion, \cite{B}, 2.2]
Suppose that $\Phi$ is a complete simplicial fan in $V$; let $n:= \dim V$. Then, $\dim_\mathbb{R} \Hom_{A_V}(H^0(\Phi;\mathcal{A}_\Phi), A_V(-2n)) = 1$.
\end{thm}

Let $\rho\in\Phi(1)$. A \emph{Courant function} of $\rho$ is a cone-wise linear function $\phi_\rho\in H^0(\Phi;\mathcal{A}_\Phi)$ which is supported on $\Star(\rho)$ and positive on the relative interior of $\rho$.

Note that, by definition, any two Courant functions of $\rho$ are positive multiples of each other. If $\Phi$ is unimodular, then \emph{the Courant function of} $\rho$ is the unique Courant function $\phi_\rho$ of $\rho$ that satisfies $\phi_\rho(v_\rho)=1$. 

Let $\tau\in\Phi$. A \emph{Courant function} $\phi_\tau$ of $\tau$ is a cone-wise polynomial function given by
\begin{equation}\label{courant function}
\phi_\tau = \prod_{\rho\in[\tau](1)} \phi_\rho ,
\end{equation}
where $\phi_\rho$ is a Courant function of $\rho\in [\tau](1)$. If $\Phi$ is unimodular, \emph{the Courant function} $\phi_\tau$ of $\tau$ is the function defined by \eqref{courant function}, where $\phi_\rho$ is the Courant function of $\rho\in[\tau](1)$.

For a top-dimensional cone $\sigma\in\Phi(n)$ let
\[
F_\sigma := \left(|\Omega_V(v_{\rho_1},\ldots, v_{\rho_n})|^{-1}\cdot \phi_\sigma\right)_\sigma \in A ,
\]
where $[\sigma](1) = \{\rho_1,\ldots,\rho_n\}$ and $v_{\rho_i}\in\rho_i$ satisfies $\phi_{\rho_i}(v_{\rho_i}) = 1$.

For $f\in H^0(\Phi;\mathcal{A}_\Phi)$, $f\colon \Phi(n)\ni\sigma\mapsto f_\sigma\in A$, let
\[
\zeta_\Phi(f) := \sum_{\sigma\in\Phi(n)} \dfrac{f_\sigma}{F_\sigma} .
\]
Then, $\zeta_\Phi(f) \in A_V$ and the assignment $f\mapsto \zeta_\Phi(f)$ defines a map
\begin{equation}\label{brion's functional}
\zeta_\Phi \colon H^0(\Phi;\mathcal{A}_\Phi) \to A_V(-2n)
\end{equation}
which satisfies $\zeta_\Phi(F_\sigma) = 1$ and, in particular, is a non-zero element of $\Hom_{A_V}(H^0(\Phi;\mathcal{A}_\Phi), A_V(-2n))$.

The map \eqref{brion's functional} induces the map of graded vector spaces
\begin{equation}\label{integral}
\int_\Phi \colon H(\Phi) \to \mathbb{R}(-2n)
\end{equation}
and, hence, an isomorphism $\int_\Phi \colon H^{2n}(\Phi) \to \mathbb{R}$.

\subsection{The Poincar\'e pairing}
The Poincar\'e pairing in cohomology of complete simplicial fans is the subject of the following theorem of \cite{B} 2.4.

\begin{thm}
The $A_V$-bilinear pairing
\[
H^0(\Phi;\mathcal{A}_\Phi)\times H^0(\Phi;\mathcal{A}_\Phi) \to A_V(-2n) \colon (f,\ g) \mapsto \zeta_\Phi(f\cdot g)
\]
induces an isomorphism $H^0(\Phi;\mathcal{A}_\Phi) \to \Hom_{A_V}(H^0(\Phi;\mathcal{A}_\Phi), A_V(-2n))$.
\end{thm}
\begin{cor}
The Poincar\'e pairing
\[
H^{2i}(\Phi)\otimes H^{2(n-i)}(\Phi) \xrightarrow{\cup} H^{2n}(\Phi) \xrightarrow{\int_\Phi} \mathbb{R}
\]
is nondegenerate.
\end{cor}
 
\subsection{Poincar\'e pairings for stars.}(See 3.6 of \cite{HLIJ}) 
Let $\tau\in\Phi$ be a $k$-dimensional cone. 

Since $\Phi$ is simplicial, it follows that $\Phi$ has a local product structure at $\tau$. Let $p:V\to V/\Span(\tau)$ denote the projection. Note that the image under $p$ of the fan
\[\{\sigma\in \bdstar(\tau): \sigma\cap\tau =\oo\}\]
is a fan in $V/\Span(\tau)$ that coincides with $\overline{\Star_{\Phi}}(\tau)$.

It is shown in Lemma 3.20 of \cite{HLIJ} that the map of algebras
\[
p^* \colon \mathcal{A}_{\overline{\Star_{\Phi}}(\tau)}(\overline{\Star_{\Phi}}(\tau))\to \mathcal{A}_{\Phi}([\Star_{\Phi}(\tau)]).
\]
induced by the projection $p$ gives rise to an isomorphism
\begin{equation}\label{mapanatural}
p^* \colon H(\overline{\Star_{\Phi}}(\tau)) \to \mathcal{A}_{\Phi}([\Star_{\Phi}(\tau)])/A^+\mathcal{A}_{\Phi}([\Star_{\Phi}(\tau)])=: H([\Star_{\Phi}(\tau)]).
\end{equation}
The restriction map $\mathcal{A}_{\Phi}(\Phi)\to \mathcal{A}_{\Phi}([\Star_{\Phi}(\tau)])$ is surjective and therefore induces a surjective map
\begin{equation}\label{restriction map}
H(\Phi) \to H([\Star_{\Phi}(\tau)]) .
\end{equation}
Let
\[
i^* \colon H(\Phi)\to H(\overline{\Star}_\Phi(\tau))
\]
denote the composition of \eqref{restriction map} with the inverse of the isomorphism \eqref{mapanatural}.

Note that $\Star_\Phi(\tau) = [\Star_\Phi(\tau)]-\bdstar_{\Phi}(\tau)$. The restriction maps
\[
r \colon \mathcal{A}_{\Phi}(\Phi)\to \mathcal{A}_{\Phi}([\Star_{\Phi}(\tau)])\text{\ \ and\ \ } r^\prime \colon \mathcal{A}_{\Phi}(\Phi)\to \mathcal{A}_{\Phi}(\Phi\setminus\Star(\tau))
\]
induce an isomorphism
\[
\Gamma_{\Star_{\Phi}(\tau)}\mathcal{A}_\Phi :=  \ker(r^\prime) \xrightarrow{\cong} \ker\left( \mathcal{A}_{\Phi}([\Star_\Phi(\tau)]) \to \mathcal{A}_{\Phi}(\bdstar_\Phi(\tau)) \right)
\]
whose inverse is given by extension by zero. Let
\[
H([\Star_{\Phi}(\tau)],\bdstar_\Phi(\tau)):= \Gamma_{\Star_{\Phi}(\tau)}\mathcal{A}_\Phi/A^+\Gamma_{\Star_{\Phi}(\tau)}\mathcal{A}_\Phi.
\]

The fan $[\Star_\Phi(\tau)]$ is quasi-convex (cf. \ref{subsection: quasi-conves fans}). 
By  Proposition 5.2 of \cite{HLIJ} the choice of the volume form $\Omega_V$ on $V$ determines non-degenerate pairings
\begin{equation}\label{dualitypairing}
[\cdot,\cdot]_{[\Star_\Phi(\tau)]} \colon\mathcal{A}_{\Phi}([\Star_\Phi(\tau)])\times\Gamma_{\Star_{\Phi}(\tau)}\mathcal{A}_\Phi \to A(2n)
\end{equation}
and
\[
(\cdot, \cdot)_{[\Star_\Phi(\tau)]} \colon H([\Star_\Phi(\tau)]) \times H([\Star_{\Phi}(\tau)],\bdstar_\Phi(\tau))\to \mathbb{R}(2n) .
\]

According to Proposition 7.1 of \cite{HLIJ}, for $a\in \mathcal{A}_\Phi(\Phi)$ and $b \in \Gamma_{\Star_\Phi(\tau)}\mathcal{A}_\Phi$
\[
[r(a),b]_{[\Star_\Phi(\tau)]} = [a,b]_\Phi .
\]

Let $\rho_1,\ldots,\rho_k$ be the rays of (the $k$-dimensional cone) $\tau$. Let $f_i\in\dual{V}$ be a linear functional positive on the relative interior of $\rho_i$. Let $\phi_i$ be a Courant function of $\rho_i$ such that $\phi_i\vert_{\rho_i} = f_i\vert_{\rho_i}$. Let $\Omega^\prime$ be a volume form on $V/\Span(\tau)$ such that $\Omega_V = f_1\wedge\cdots\wedge f_k\wedge\Omega^\prime$.

Observe that $\phi_\tau := \phi_1\cdot\ldots\cdot\phi_k$ is a Courant function for $\tau$ and that multiplication by $\phi_\tau$ induces maps
\[
\phi_\tau\cdot \colon \mathcal{A}_\Phi(\Phi) \to\Gamma_{\Star_\Phi(\tau)}\mathcal{A}_\Phi \text{\ \ and\ \ } \phi_\tau\cdot \colon \mathcal{A}_\Phi([\Star_\Phi(\tau)]) \to\Gamma_{\Star_\Phi(\tau)}\mathcal{A}_\Phi
\]
which commute with the restriction map $r:\mathcal{A}_\Phi(\Phi)\to \mathcal{A}_\Phi([\Star_\Phi(\tau)])$. It follows from Section 7.3 and Proposition 7.8 of \cite{HLIJ} that, for $a,b\in H(\Phi)$,
\begin{equation}\label{profor}
\int_{\overline{\Star}_\Phi(\tau)}i^*a\cdot i^*b = n\cdot (r(a),\phi_\tau\cdot r(b))_{[\Star_\Phi(\tau)]} = \int_{\Phi}a\cdot(\phi_\tau\cdot b) .
\end{equation}
According to Corollary 7.9 of \cite{HLIJ}, multiplication by $\phi_\tau$ induces an isomorphism
\[
\phi_\tau \cdot: H([\Star_\Phi(\tau)]) \to H([\Star_\Phi(\tau)],\bdstar_\Phi(\tau))(2k) .
\]
Since $[\Star_\Phi(\tau)]$ is quasi-convex, the natural inclusion $\Gamma_{\Star_\Phi(\tau)}\mathcal{A}_\Phi \to \mathcal{A}_\Phi(\Phi)$ gives rise to an injective map
\begin{equation}\label{inclusionnatural}
H([\Star_\Phi(\tau)],\bdstar_\Phi(\tau)) \to H(\Phi) .
\end{equation}

\subsection{The Gysin map for stars.}
Suppose that $\tau$ is a $k$-dimensional cone of $\Phi$.

The \emph{Gysin map}
\[
i_*:H(\overline{\Star}_\Phi(\tau))\to H(\Phi)(2k)
\]
is defined as the composition
\[
H(\overline{\Star}_\Phi(\tau))\xrightarrow{\eqref{mapanatural}} H([\Star_\Phi(\tau)]) \xrightarrow{\phi_\tau\cdot} H([\Star_\Phi(\tau)],\bdstar_\Phi(\tau))(2k) \xrightarrow{\eqref{inclusionnatural}} H(\Phi)(2k) .
\]

\begin{prop}
{~}
\begin{enumerate}\label{prop: gysinstars}
\item The composition $i_*\circ i^*$ coincides with multiplication by $\phi_\tau$.

\item The Gysin map satisfies the projection formula. Namely, for $a\in H(\Phi)$ and $b\in H(\overline{\Star}_\Phi(\tau))$
\[
i_*(i^*a\cdot b) = a\cdot i_*b .
\]

\item For $a\in H(\overline{\Star}_\Phi(\tau))$
\[
\int_{\overline{\Star}_\Phi(\tau)}a = \int_\Phi i_*a.
\]
\end{enumerate}
\end{prop}
\begin{proof}
{~}
\begin{enumerate}
\item Clear from the definition of the Gysin map.

\item Let $a\in H(\Phi)$ and $b\in H(\overline{\Star}_\Phi(\tau))$ be fixed. For any $c\in H(\Phi)$ the equality \eqref{profor} shows that
\[
\int_{\Phi}c\cdot i_*(i^*a\cdot b) = \int_{\overline{\Star}_\Phi(\tau)}i^*c\cdot i^*a\cdot b
        = \int_{\overline{\Star}_\Phi(\tau)}i^*\left(c\cdot a\right)\cdot b
        = \int_{\Phi}c\cdot a\cdot i_*b
        = \int_{\Phi}c\cdot \left(a\cdot i_*b\right).
\]
Since the Poincar\'e pairing is non-degenerate, it follows that $i_*(i^*a\cdot b) = a\cdot i_*b$.

\item The statement follows directly from \eqref{profor}.
\end{enumerate}
\end{proof}

\subsection{Decomposition and the Gysin map for subdivisions}\label{subsection: Decomposition and the Gysin map} 

\begin{thm}[M.~Brion, \cite{B}, 2.3]
Suppose that $\pi\colon \Psi \to \Phi$ is a simplicial subdivision of $\Phi$. Then, there exists a unique $H^0(\Phi;\mathcal{A}_\Phi)$-linear map
\begin{equation}\label{brion's gysin}
\pi_* \colon H^0(\Phi;\mathcal{A}_\Psi) \to H^0(\Phi;\mathcal{A}_\Phi)
\end{equation}
such that $\pi_*(1) = 1$. Namely, for $f\in H^0(\Phi;\mathcal{A}_\Psi)$ and a maximal cone $\sigma\in\Phi$
\begin{equation*}
\pi_*(f)_\sigma = F_\sigma\sum_{\tau\in\pi^{-1}([\sigma])(n)} \dfrac{f_\tau}{F_\tau}
\end{equation*}
\end{thm}

The maps \eqref{brion's gysin} induces the \emph{Gysin map} of graded vector spaces (which is denoted in the same way as it is clear from context)
\begin{equation}\label{gysin map}
\pi_* \colon H(\Psi) \to H(\Phi).
\end{equation}

The maps \eqref{brion's gysin} and \eqref{gysin map} satisfy
\begin{equation*}
\pi_*\circ\pi^* = \id,\ \ \ \zeta_\Psi = \zeta_\Phi\circ\pi_*,\ \ \ \int_\Psi = \int_\Phi\circ\pi_* .
\end{equation*}

Therefore, $\pi^*$ is injective, $\pi_*$ is surjective and there is a canonical direct sum decomposition
\begin{equation}\label{decomposition under subdivision}
H(\Psi) = \im(\pi^*)\oplus\ker(\pi_*) .
\end{equation}

\begin{lemma}\label{lemma: orthognal decomposition}
The decomposition \eqref{decomposition under subdivision} is orthogonal with respect to the Poincar\'e pairing.
\end{lemma}
\begin{proof}
Let $\alpha\in H(\Phi)$, $\beta\in\ker(\pi_*)$. Then,
\begin{equation*}
\int_\Psi \pi^*(\alpha)\cdot\beta = \int_\Phi\pi_*(\pi^*(\alpha)\cdot\beta) = \int_\Phi\alpha\cdot\pi_*(\beta) = 0 .
\end{equation*}
\end{proof}
\subsection{Gysin maps for star subdivisions.} 
Suppose that $\pi: \Psi \to \Phi$ is a star subdivision at $\tau\in\Phi(k)$ along a ray $\rho\in\Psi$. In what follows we shall refer to the following diagram:
\begin{equation}
\begin{tikzcd}
\Star(\rho) \arrow{r}{\widetilde{i}}\arrow{d}{\widetilde{\pi}}& \Psi\arrow{d}{\pi} \\
 \Star(\tau) \arrow{r}{i} &\Phi
\end{tikzcd}.
\end{equation}
Let $\Upsilon$ denote the fan in $V$ given by
\begin{equation*}
    \Upsilon :=\Phi \backslash \Star_\Phi(\tau) = \Psi\backslash \Star_\Psi(\rho),
\end{equation*} 
and let $\Upsilon^0 := \Upsilon\backslash \bdstar_\Phi(\tau) = \Upsilon\backslash \bdstar_\Psi(\rho)$.
The fan $\Upsilon$ is quasi-convex because its boundary is the same as that of $[\Star_\Phi(\tau)]$ and $[\Star_\Psi(\rho)]$. It is clear that $\mathcal{A}_\Phi(\Upsilon) = \mathcal{A}_\Psi(\Upsilon)$. In addition, let
\begin{align*}
    \Gamma_{\Upsilon^0}\mathcal{A}_\Phi &:=\ker\left(\mathcal{A}_\Phi(\Upsilon)\to \mathcal{A}_\Phi(\bdstar_\Phi(\tau)\right)=\ker\left(\mathcal{A}_\Psi(\Upsilon)\to \mathcal{A}_\Phi(\bdstar_\Psi(\rho)\right). 
\end{align*}
 and let 
    \[H(\Upsilon,\partial\Upsilon) := \Gamma_{\Upsilon^0}\mathcal{A}_\Phi/A^+\Gamma_{\Upsilon^0}\mathcal{A}_\Phi.\]
Observe that $\pi^*$ induces a morphism
\begin{equation*}
    {\pi}^*:\mathcal{A}_\Phi([\Star_\Phi(\tau)]) \to \mathcal{A}_\Psi([\Star_\Psi(\rho)]),
\end{equation*}
and its restriction to $\Gamma_{\Upsilon^0}\mathcal{A}_\Phi$ is just the identity map. Hence, $\pi^*$ gives rise to the commutative diagram with exact rows
\[
\begin{CD}
0 @>>> H(\Upsilon,\partial\Upsilon) @>>> H(\Phi) @>>> H([\Star_\Phi(\tau)]) @>>> 0 \\
& & @| @VV{\pi^*}V @VV{\pi^*}V \\
0 @>>> H(\Upsilon,\partial\Upsilon) @>>> H(\Psi) @>>> H([\Star_\Psi(\rho)]) @>>> 0
\end{CD}
\]
where the horizontal morphisms are the natural inclusions and restrictions.
Therefore, $\pi^*$ induces a map of algebras
\begin{equation}
    \widetilde{\pi}^*: H(\overline{\Star}_\Phi(\tau))\to H(\overline{\Star}_\Psi(\rho)).
\end{equation}
On the other hand, the commutative diagram with exact rows given by the natural inclusions and restrictions
\[
\begin{CD}
0 @>>> \Gamma_{\Star_\Psi(\rho)}\mathcal{A}_\Psi @>>> \mathcal{A}_\Psi(\Psi) @>>> \mathcal{A}_\Psi(\Upsilon) @>>> 0 \\
& & @VV{\pi_*}V @VV{\pi_*}V @| \\
0 @>>> \Gamma_{\Star_\Phi(\tau)}\mathcal{A}_\Phi @>>> \mathcal{A}_\Phi(\Phi) @>>> \mathcal{A}_\Psi(\Upsilon) @>>> 0
\end{CD}
\]
shows that the restriction of $\pi_*$ to $\Gamma_{\Star_\Psi(\rho)}\mathcal{A}_\Psi$ gives rise to a surjective map
\begin{equation}
    {\pi}_*:H([\Star_\Psi(\rho)],\bdstar_\Psi(\rho))\to H([\Star_\Phi(\tau)],\bdstar_\Phi(\tau)).
\end{equation}
In particular, $\pi_*$ gives rise to the following commutative diagram with exact rows:
\[
\begin{CD}
0 @>>> H([\Star_\Psi(\rho)],\bdstar_\Psi(\rho)) @>>> H(\Psi) @>>> H(\Upsilon) @>>> 0 \\
& & @VV{\pi_*}V @VV{\pi_*}V @| \\
0 @>>> H([\Star_\Phi(\tau)],\bdstar_\Phi(\tau)) @>>> H(\Phi) @>>> H(\Upsilon) @>>> 0
\end{CD}
\]
The above diagram shows that the map 
\[
\pi_* \circ \widetilde{i}_*:H(\overline{\Star}_\Psi(\rho))(2) \to H(\Phi), 
\]
factors through $i_*:H(\overline{\Star}_\Phi(\tau))(2k)\to H(\Phi)$ and $\pi_*$ induces the \emph{Gysin map}
\begin{equation}\label{gysin map star subdivision}
    \widetilde{\pi}_*: H(\overline{\Star}_\Psi(\rho))(2(1-k)) \to H(\overline{\Star}_\Phi(\tau)) .
\end{equation}

\begin{prop}
{~}
\begin{enumerate}
\item The Gysin map \eqref{gysin map star subdivision} satisfies the projection formula. Namely, for  $a\in H(\overline{\Star}_\Psi(\rho))$ and $b\in H(\overline{\Star}_\Phi(\tau))$,
\[
\widetilde{\pi}_*\left(a\cdot \widetilde{\pi}^*b\right) = \left(\widetilde{\pi}_*a\right) \cdot b .
\]

\item For $a\in H(\overline{\Star}_\Psi(\rho))$
\[
\int_{\overline{\Star}_\Phi(\tau)}\widetilde{\pi}_*a = \int_{\overline{\Star}_\Psi(\rho)} a .
\]
\end{enumerate}
\end{prop}
\begin{proof}
{~}
\begin{enumerate}
\item Let $a\in H(\overline{\Star}_\Psi(\rho))$ and $b\in H(\overline{\Star}_\Phi(\tau))$. Since $i_*$ is injective, it suffices to show that $i_*\left(\widetilde{\pi}_*\left(a\cdot \widetilde{\pi}^*b\right) \right)= i_*\left(\left(\widetilde{\pi}_*a\right) \cdot b\right)$. Since $i^*$ is surjective, there exists $\beta\in H(\Phi)$ such that $i^*\beta =b$. Then,
\begin{multline*}
i_*\left(\widetilde{\pi}_*\left(a\cdot \widetilde{\pi}^*b\right)\right) = \pi_*\left(\widetilde{i}_*\left(a\cdot \widetilde{\pi}^*i^*\beta\right)\right) = \pi_*\left(\widetilde{i}_*\left(a\cdot \widetilde{i}^*\pi^*\beta\right)\right) = \pi_*\left(\left(\widetilde{i}_*a\right)\cdot\pi^*\beta\right) \\
=\left(\pi_*\widetilde{i}_*a\right)\cdot \beta =\left(i_* \widetilde{\pi}_*a\right)\cdot \beta =i_*\left((\widetilde{\pi}_*a)\cdot i^*\beta\right)=i_*\left((\widetilde{\pi}_*a)\cdot b\right).
\end{multline*}

\item It follows from Proposition \ref{prop: gysinstars} (3) and $i_*\widetilde{\pi}_* = \pi_*\widetilde{i}_*$ that
\[
    \int\limits_{\overline{\Star}_\Phi(\tau)}\!\!\!\!\! \widetilde{\pi}_*a =\int\limits_{\Phi}i_*\widetilde{\pi}_*a=\int\limits_{\Phi}\pi_*\widetilde{i}_*a= \int\limits_{\Psi} \widetilde{i}_*a =\int\limits_{\overline{\Star}_\Psi(\rho)}\!\!\!\!\! a.
\]
\end{enumerate}
\end{proof}

Let $\psi_\rho\in\mathcal{A}_\Psi(\Psi)$ denote a Courant function of $\rho$ and let $c:= \widetilde{i}^*\psi_\rho \in H^2(\Star(\rho))$. The map $\widetilde{\pi}^*\colon H(\Star(\tau)) \to H(\Star(\rho))$ induces an isomorphism $H(\Star(\tau))[c]/(p(c)) \xrightarrow{\cong} H(\Star(\rho))$, where $p(c)$ is a (monic) polynomial of degree $\dim\tau$.
\begin{prop}\label{exceptionaldivisor}
    In the situation above $(-1)^{k-1}\widetilde{\pi}_*(c^{k-1})$ is positive.
\end{prop}
\begin{proof}
Let $\xi\in[\Star_{\Phi}(\tau)]$ be such that $\xi\cap\tau=\oo$ and $\tau+\xi\in \Phi(n)$. Let $\phi_\tau, \phi_\xi\in \mathcal{A}_\Phi(\Phi)$ denote Courant functions for $\tau$ and $\xi$ such that
\[
\int_{\Phi}\phi_\xi\cdot\phi_\tau = 1 .
\]
It follows that
\[
\int\limits_{\overline{\Star}_\Phi(\tau)}\!\!\!\!\! i^*\phi_\xi = \int\limits_{\Phi}\phi_\tau\cdot\phi_\xi = 1 .
\]    
Observe that
\begin{multline*}
\widetilde{\pi}_*(c^{k-1}) =\int_{\overline{\Star}_{\Phi}(\tau)} \widetilde{\pi}_*(c^{k-1})\cdot i^*\phi_\xi = \int_{\Phi} i_*\left(\widetilde{\pi}_*(c^{k-1})\cdot i^*\phi_\xi\right) \\
        =\int_{\Phi} \left(i_*\widetilde{\pi}_*(c^{k-1})\right)\cdot \phi_\xi
        =\int_{\Phi} \left(\pi_*\widetilde{i}_*(c^{k-1})\right)\cdot \phi_\xi
        =\int_{\Phi} \pi_*\left(\widetilde{i}_*(c^{k-1})\cdot \pi^*\phi_\xi\right)\\
        =\int_\Psi\widetilde{i}_*(c^{k-1})\cdot \pi^*\phi_\xi
        =\int_\Psi\psi_\rho^k\cdot \pi^*\phi_\xi .
\end{multline*}
Therefore, it is sufficient to show that
\[
(-1)^{k-1}\cdot\int_\Psi\psi_\rho^k\cdot \pi^*\phi_\xi > 0 .
\]

Let $\eta$ be a facet of $\tau$, and let $\rho_1,\dots,\rho_{k-1}$ be the (distinct) rays of $\eta$. For $1\leq j\leq k-1$ let $\psi_j \in \mathcal{A}_\Psi(\Psi)$ denote the Courant function of $\rho_j$ such that $\psi_j\vert_{\rho_j} = \phi_j\vert_{\rho_j}$. Then,
\[
\psi_{\eta} := \prod_{j=1}^{k-1}\psi_j \in \mathcal{A}_\Psi(\Psi)
\]
is a Courant function of $\eta\in\Psi$. Since $\xi+\rho$ is a cone of $\Psi$ and $\pi^*\phi_\xi\in\mathcal{A}_\Psi(\Psi)$ is a Courant function of $\xi$, it follows that $(\pi^*\phi_\xi)\cdot\psi_{\eta}\cdot \psi_\rho$ is a Courant function of the top-dimensional cone $\xi+\eta+\rho$ of $\Psi$ and, therefore,
\[
\int_{\Psi}(\pi^*\phi_\xi)\cdot\psi_{\eta}\cdot \psi_\rho>0.
\]

For $1\leq j\leq k-1$ let $L_j$ denote a linear functional which
\begin{itemize}
\item is positive on the relative interior of $\rho_j$, and
\item vanishes on the facet of $\tau+\xi$ opposite to $\rho_j$.
\end{itemize}
Each linear functional $L_j$ gives rise to a linear relation of the form
\[
-\psi_\rho = \lambda_j \psi_j + \left(\textnormal{Courant functions of rays not in } [\tau+\xi] (1) \cup \{\rho\}\right),
\]
in $\mathcal{A}_\Psi(\Psi)$, where $\lambda_j$ is a positive real number. 
Therefore, the equality
\[
\psi_\rho^k\cdot\pi^*\phi_\xi=(-1)^{k-1}\left(\prod_{j=1}^{k-1}\lambda_j\right)\cdot \psi_{\eta}\cdot \psi_\rho\cdot \pi^*\phi_\xi.
\]
holds in $H(\Psi)$. Therefore,
\[
(-1)^{k-1}\widetilde{\pi}_*(c^{k-1}) = (-1)^{k-1}\int_\Psi\psi_\rho^k\cdot \pi^*\phi_\xi =\left(\prod_{j=1}^{k-1}\lambda_j\right)\cdot \int_\Psi \psi_{\eta}\cdot \psi_\rho\cdot \pi^*\phi_\xi > 0.
\]
\end{proof}

\subsection{Integration of convex monomials}
Recall that a piecewise linear function $\ell$ on $\Phi$ is \emph{convex} (respectively, \emph{strictly convex}) if for any two top-dimensional cones $\sigma$ and $\tau$ and any vector $v$ in the interior of $\tau$ the inequality $\ell_\sigma(v)\leqslant \ell_\tau(v)$ (resp. $\ell_\sigma(v)<\ell_\tau(v)$) holds.

\begin{lemma}\label{lemma: addition of linear function preserves convexity}
If $\ell\in H^0(\Phi;\mathcal{A}_\Phi)$ is a (strictly) convex piecewise linear function, then for any $L\in V^\vee$ the piecewise linear function $L+\ell$ is (strictly) convex.
\end{lemma}
\begin{proof}
Suppose that $\sigma$ and $\tau$ are top-dimensional cones of $\Phi$ and $v\in V$ belongs to the interior of $\tau$. Since $L$ is a linear functional on $V$, it follows that $L_\sigma = L_\tau$. Hence, 
\[
(\ell+L)_\sigma (v) = \ell_\sigma(v) + L_\sigma(v) \leqslant \ell_\tau(v) + L_\tau(v)=(\ell+L)_\tau ( v).
\]
If $\ell$ is strictly convex, then the above inequality is strict.
\end{proof}

\begin{lemma}\label{lemma: pullback preserves convexity}
If $\pi:\Psi\to \Phi$ is a subdivison and $\ell\in H^0(\Phi;\mathcal{A}_\Phi)$ is convex, then $\pi^*\ell$ is convex.
\end{lemma}
\begin{proof}
Suppose that $\sigma$ and $\tau$ are top-dimensional cones of $\Psi$ and $v\in V$ belongs to the interior of $\tau$. Then, $\pi(\sigma)$ and $\pi(\tau)$ are are top-dimensional cones of $\Phi$ and $v$ belongs to the interior of $\pi(\tau)$. 

Since $(\pi^*\ell)_{\sigma} = \ell_{\pi(\sigma)}$ and $(\pi^*\ell)_\tau = \ell_{\pi(\tau)}$ it follows that
\[
(\pi^*\ell)_{\sigma}(v) = \ell_{\pi(\sigma)}(v)\leqslant \ell_{\pi(\tau)}(v) = (\pi^*\ell)_\tau(v).
\]
\end{proof}

\begin{prop}\label{prop: nonnegative convex monomials}
Let $n = \dim \Phi$. Suppose $\ell_1,\dots,\ell_n\in H^0(\Phi;\mathcal{A}_\Phi)$ are convex piecewise linear functions, then
\[
\int_\Phi \ell_1\cdot\ldots\cdot \ell_n \geqslant 0.
\]
\end{prop}
\begin{proof}
Consider a projective subdivision $\pi:\Psi\to \Phi$. By Lemma \ref{lemma: pullback preserves convexity} the functions $\pi^* \ell_i$, $i = 1,\ldots,n$, are convex on $\Psi$.
Since
\[
\int_\Phi \ell_1\cdot\ldots\cdot \ell_n = \int_{\Phi}\pi_*\pi^*\left(\ell_1\cdot\ldots\cdot \ell_n\right)
= \int_\Psi \pi^*\left(\ell_1\cdot\ldots\cdot \ell_n\right)
= \int_\Psi \pi^*\ell_1\cdot \ldots\cdot \pi^*\ell_n.
\]
it follows that we may assume without loss of generality that $\Phi$ is projective. Let $\ell$ denote a strictly convex piecewise linear function on $\Phi$. Observe that for any $t > 0$, the functions $\ell_i(t) := \ell_i+t\cdot \ell$, $i = 1,\ldots,n$, are strictly convex. 

The assignment  
\[
t \mapsto \int_\Phi \ell_1(t)\cdot\ldots\cdot\ell_n(t)
\]
is a polynomial function of $t$ degree $n$ and 
\[
\int_\Phi\ell_1\cdot\ldots\cdot \ell_n = \lim\limits_{t\to 0} \int_\Phi \ell_1(t)\cdot\ldots\cdot\ell_n(t).
\]
Therefore, it is sufficient to show that 
\[
\int_\Phi\ell_1\cdot\ldots\cdot\ell_n \geqslant 0
\]
if $\ell_1,\ldots,\ell_n$ are strictly convex. Under the latter assumption the integral in question is a positive multiple of the mixed volume of the corresponding polytopes by Corollary 5.3 of \cite{B}, hence non-negative.
\end{proof}

\section{Signature of complete simplicial fans}\label{section: signature}
Suppose that $\Phi$ is a complete simplicial fan.

\subsection{Signature of simplicial fans}
We denote by $\sign(\Phi)$ the signature of the non-degenerate symmetric Poincar\'e pairing
\[
H(\Phi)\otimes H(\Phi) \to \mathbb{R}\ \ \colon \alpha\otimes\beta\mapsto \int_\Phi\alpha\cdot \beta.
\]

Note that
\begin{itemize}
\item the signature of an odd-dimensional fan is equal to zero;
\item if $\dim\Phi = 2m$, then $\sign(\Phi)$ coincides with the signature of the restriction of the Poincar\'e pairing to the middle cohomology $H^{2m}(\Phi)$.
\end{itemize}

\begin{prop}\label{prop: subdivision signature}
Suppose that $\dim\Phi$ is even and $\pi \colon \Psi \to \Phi$ is a star subdivision at $\tau\in\Phi$. Then, $\sign(\Psi) = \sign(\Phi) - \sign(\overline{\Star}(\tau))$.
\end{prop}
\begin{proof}
Let $n := \dim \Phi = \dim \Psi$, let $k := \dim \tau$. Let $\phi_\rho \in\mathcal{A}_\Psi(\Psi)$ denote a Courant function for $\rho$ and let $c:= \widetilde{i}^*\phi_\rho \in H^2(\overline{\Star}(\rho))$. The map $\widetilde{\pi}^*\colon H(\overline{\Star}(\tau)) \to H(\overline{\Star}(\rho))$ induces an isomorphism $H(\overline{\Star}(\tau))[c]/(p(c)) \xrightarrow{\cong} H(\overline{\Star}(\rho))$, where $p(c)$ is a polynomial of degree $\dim\tau$.

The decomposition
\begin{equation}\label{dir im decomp}
\pi_*\mathcal{A}_\Psi \cong \mathcal{A}_\Phi \oplus \mathcal{M},\ \ \mathcal{M} = \bigoplus\limits_{i=1}^{k-1} \left.\mathcal{A}_\Phi\right\vert_{\Star(\tau)}(-2i).
\end{equation}
is preserved by the self-duality isomorphism induced by that of $\mathcal{A}_\Psi$. Note $\left.\mathcal{A}_\Phi\right\vert_{\Star(\tau)}(-2i)$ is self-dual if and only if $k=2i$. The decomposition \eqref{dir im decomp} induces the orthogonal (with respect to the Poincar\'e pairing) decomposition
\[
H(\Psi) \cong H(\Phi)\oplus\overline{\Gamma(\Phi;\mathcal{M})}
\]
which implies that $\sign(\Psi) = \sign(\Phi)+\sign(\overline{\Gamma(\Phi;\mathcal{M})})$.

If $k$ is odd, then
\begin{itemize}
\item  $\mathcal{M}$ has no self-dual summands and, therefore, the signature of the Poincar\'e pairing on $\overline{\Gamma(\Phi;\mathcal{M})}$ is equal to zero
\item $n-k = \dim\overline{\Star}(\tau)$ is odd, hence $\sign(\overline{\Star}(\tau)) = 0$.
\end{itemize}
Hence, if $k$ is odd, the claimed equality holds. 

Suppose that $k$ is even. Then, $\mathcal{M}$ contains a unique self-dual summand, namely $\left.\mathcal{A}_\Phi\right\vert_{\Star(\tau)}(-k)$, hence $\sign(\overline{\Gamma(\Phi;\mathcal{M})}) = \pm\sign(\overline{\Star}(\tau))$.

The part of $H(\Psi)$ which corresponds to the self-dual summand in the decomposition \eqref{dir im decomp} is the image of $c^{\frac{k}2 - 1}\cdot\widetilde{\pi}^*H(\overline{\Star}(\tau))$ under the map $\widetilde{i}_*$.

For $\alpha, \beta \in H(\overline{\Star}(\tau))$
\[
\int\limits_\Psi \widetilde{i}_*(c^{\frac{k}2 - 1}\widetilde{\pi}^*\alpha)\cdot\widetilde{i}_*(c^{\frac{k}2 - 1}\widetilde{\pi}^*\beta) =  \int\limits_\Psi \widetilde{i}_*(c^{k-1}\cdot\widetilde{\pi^*}\alpha\cdot\widetilde{\pi}^*\beta) 
 =  \int\limits_{\overline{\Star}(\rho)}\!\!\!\!\! c^{k-1}\cdot\widetilde{\pi}^*\alpha\cdot\widetilde{\pi}^*\beta 
 =  \widetilde{\pi}_*(c^{k-1})\int\limits_{\overline{\Star}(\tau)}\!\!\!\!\! \alpha\cdot\beta
\]
Since $k$ is even, it follows from Proposition \ref{exceptionaldivisor} that $\widetilde{\pi}_*(c^{k-1}) < 0$. Thus, $\sign(\overline{\Gamma(\Phi;\mathcal{M})}) = -\sign(\overline{\Star}(\tau))$ and $\sign(\Psi) = \sign(\Phi) - \sign(\overline{\Star}(\tau))$.
\end{proof}

\subsection{Signature is Euler characteristic}
For a complete fan $\Phi$ let
\[
\varepsilon(\Phi) := \sum\limits_i (-1)^i\dim H^{2i}(\Phi) .
\]
By Proposition \ref{prop: coh of forms}
\[
\varepsilon(\Phi) = \sum\limits_i (-1)^i\dim H^i(\Phi;\Omega^i_\Phi) = \sum\limits_i \chi(\Phi;\Omega^i_\Phi) .
\]

\begin{prop}\label{prop: subdivision euler}
Suppose that $\dim\Phi$ is even and $\pi \colon \Psi \to \Phi$ is a star subdivision at $\tau\in\Phi$. Then, $\varepsilon(\Psi) = \varepsilon(\Phi) - \varepsilon(\overline{\Star}(\tau))$.
\end{prop}
\begin{proof}
Let $n := \dim \Phi = \dim \Psi$, let $k := \dim \tau$. The decomposition \eqref{dir im decomp} gives rise to the isomorphism
\[
H(\Psi) \cong H(\Phi) \oplus \bigoplus\limits_{i=1}^{k-1} H(\overline{\Star}(\tau))(-2i).
\]
Therefore,
\[
\varepsilon(\Psi) = \varepsilon(\Phi) + \sum\limits_{i=1}^{k-1} (-1)^i \varepsilon(\overline{\Star}(\tau)).
\]
If $k$ is odd, then 
\begin{itemize}
\item $\sum\limits_{i=1}^{k-1} (-1)^i \varepsilon(\overline{\Star}(\tau)) = 0$
\item $n-k = \dim\overline{\Star}(\tau)$ is odd, hence $\varepsilon(\overline{\Star}(\tau)) = 0$.
\end{itemize}
If $k$ is even, then $\sum\limits_{i=1}^{k-1} (-1)^i \varepsilon(\overline{\Star}(\tau)) = -\varepsilon(\overline{\Star}(\tau))$. Hence, in either case the claimed equality holds.
\end{proof}

\begin{thm}\label{thm: signature is euler characteristic}
For any complete simplicial fan $\Phi$
\[
\sign(\Phi) = \varepsilon(\Phi) .
\]
\end{thm}
\begin{proof}
If $\dim\Phi$ is odd, then $\sign(\Phi) = \varepsilon(\Phi) = 0$. Hence we shall assume that $\dim\Phi$ is even.

We leave it to the reader to verify the claimed equality for $\Phi = \mathbb{P}^n$ and all $n$. We proceed by induction on the dimension of the fan.

Suppose that $\Phi$ is a complete simplicial fan of dimension $n$. By Theorem \ref{thm: subdivision theorem} there exists a diagram of star subdivisions
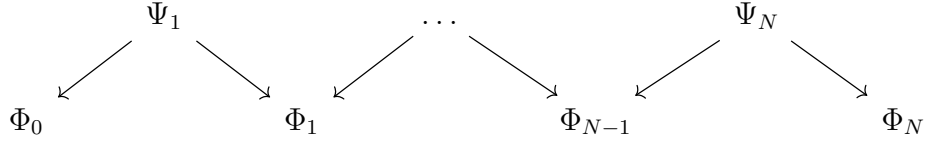
\begin{equation*}
\begin{tikzcd}
 & \Psi_1 \arrow[ld] \arrow[rd] & & \dots \arrow[ld] \arrow[rd] & & \Psi_N \arrow[ld] \arrow[rd] & \\
\Phi_0 & & \Phi_1 & & \Phi_{N-1} & & \Phi_N
\end{tikzcd}
\end{equation*}
with $\Phi_0 = \mathbb{P}^n$ and $\Phi_N = \Phi$. Thus, it suffices to prove the following statement.

\vskip 3mm\noindent\emph{Claim:}
Suppose that $\pi \colon \Psi \to \Phi$ is a star subdivision at $\tau\in\Phi$ along a ray $\rho\in\Psi$. Then, $\sign(\Phi) = \varepsilon(\Phi)$ if and only if $\sign(\Psi) = \varepsilon(\Psi)$.
\vskip 3mm\noindent\emph{Proof:}
By induction on dimension $\sign(\overline{\Star}(\tau)) = \varepsilon(\overline{\Star}(\tau))$. The claim follows from Proposition \ref{prop: subdivision euler} and Proposition \ref{prop: subdivision signature}.
\end{proof} 

\section{A Riemann-Roch type theorem}\label{section: riemann-roch}
Suppose that $\Phi$ is a complete \emph{unimodular} (hence, in particular, rational) simplicial fan in $V$.

\subsection{The sheaf $\widehat{\mathcal{A}}_\Phi$}
Let $\widehat{A} = \widehat{A}_V := \prod\limits_{i = 0}^\infty A^{2i}$, $\widehat{A}^+ := \ker(\widehat{A} \to A^0)$.
Let $\widehat{\mathcal{A}}_\Phi := \prod\limits_{i = 0}^\infty\mathcal{A}^{2i}_\Phi$.
\begin{lemma}
The map
\[
H(\Phi) = \overline{H^0(\Phi;\mathcal{A}_\Phi)}\to H^0(\Phi;\widehat{\mathcal{A}}_\Phi)/\widehat{A}^+H^0(\Phi;\widehat{\mathcal{A}}_\Phi)
\]
induced by the canonical map $\mathcal{A}_\Phi \to \widehat{\mathcal{A}}_\Phi$ is an isomorphism.
\end{lemma}
\begin{proof}
Since $\Phi$ is complete and simplicial, then $\Gamma(\Phi;\mathcal{A}_\Phi)$ is a free $A$-module of finite rank. Since the reduction modulo $A^+$ of the canonical map $A\to\hat{A}$ is an isomorphism, it then follows that the map
\[
H(\Phi) \to H^0(\Phi;\widehat{\mathcal{A}}_\Phi)/\widehat{A}^+H^0(\Phi;\widehat{\mathcal{A}}_\Phi)
\]
is also an isomorphism.
\end{proof}

\subsection{The Chern character}\label{subsection: chern character}
For $\rho\in\Phi(1)$ let $v_\rho\in\rho$ denote the primitive vector. Recall that the corresponding Courant function is denoted $\phi_\rho$.

According to Proposition \ref{prop: bases for K}, the collection of sheaves $\mathcal{O}$, $\mathcal{O}(\sigma)$, $\sigma\in\Phi$ (\eqref{line bundles associated to cones}, \eqref{line bundles associated to rays}, form a basis for $K(\Phi)$. Let
\[
\widetilde{\ch} \colon K(\Phi) \to H^0(\Phi;\widehat{\mathcal{A}}_\Phi)
\]
denote the unique homomorphism such that
\begin{itemize}
\item $\widetilde{\ch}(\mathcal{O}) = 1$,
\item $\widetilde{\ch}(\mathcal{O}(\rho)) = \exp(-\phi_\rho)$ if $\rho\in\Phi(1)$,
\item $\widetilde{\ch}(\mathcal{O}(\sigma)) = \prod_{\rho\in[\sigma](1)}\widetilde{\ch}(\mathcal{O}(\rho))$.
\end{itemize}
Let
\[
\ch \colon K(\Phi) \to H(\Phi)
\]
denote the composition
\[
K(\Phi) \xrightarrow{\widetilde{\ch}} H^0(\Phi;\widehat{\mathcal{A}}_\Phi) \to H^0(\Phi;\widehat{\mathcal{A}}_\Phi)/\widehat{A}^+H^0(\Phi;\widehat{\mathcal{A}}_\Phi) \cong H(\Phi) .
\]

\subsection{Multiplicativity properties of the Chern character}
\begin{lemma}\label{lemma: mult projective space}
Let $\sigma\in\Phi$ and let $\xi\in\Phi(1)$ be a ray such that
\begin{itemize}
\item $\xi+\sigma\notin\Phi$,
\item for any $\tau\in\partial\sigma$, $\xi+\tau\in\Phi$.
\end{itemize}
 Then,
\[
\widetilde{\ch}\left([\mathcal{O}(\xi)]\cdot[\mathcal{O}(\sigma)]\right) = \widetilde{\ch}\left([\mathcal{O}(\xi)]\right)\cdot\widetilde{\ch}\left([\mathcal{O}(\sigma)]\right) .
\]
\end{lemma}
\begin{proof}
Let $n:= \dim\sigma$, let $S := \{\xi\}\cup[\sigma](1)$. Note that the hypotheses mean that the subfan $\langle S\rangle$ is combinatorialy equivalent to $\mathbb{P}^n$. For a non-empty, proper subset $I\subset S$ let $\rho_I := \sum\limits_{\rho\in I} \rho$; let $\rho_\varnothing := \oo$. Then, $\langle S\rangle = \{ \rho_I \mid I\subsetneq S\}$. Thus (see Example \ref{example: projective space}),
\[
[\mathcal{O}(\xi)]\cdot[\mathcal{O}(\sigma)] = \prod_{\rho\in S} [\mathcal{O}(\rho)] = \sum_{I\subsetneq S} (-1)^{n-|I|}[\mathcal{O}(\rho_I)] .
\]
For a non-empty subset $I\subset S$ let $\phi_I = \sum\limits_{\rho\in I} \phi_\rho$; let $\phi_\varnothing = 0$.
It is therefore sufficient to show  that
\[
\sum_{I\subseteq S} (-1)^{n-|I|}\exp(-\phi_I) = 0
\]
in $H^0(\Phi;\widehat{\mathcal{A}}_\Phi)$.

The homogeneous component of $\sum\limits_{I\subseteq S} (-1)^{n-|I|}\exp(-\phi_I)$ of degree zero is the reduced Euler characteristic of the $n$-dimensional simplex, hence equal to zero.

Let $k > 0$. The homogeneous component of $\sum\limits_{I\subseteq S} (-1)^{n-|I|}\exp(-\phi_I)$ of degree $k$ is $(-1)^k\sum\limits_{I\subseteq S} (-1)^{|I|}\phi_I^k$. Since this sum is divisible by $\prod\limits_{\rho\in S}\phi_\rho$ and the latter monomial is equal to zero in $H^0(\Phi;\widehat{\mathcal{A}}_\Phi)$, it follows that $(-1)^k\sum\limits_{I\subseteq S} (-1)^{|I|}\phi_I^k = 0$ in $H^0(\Phi;\widehat{\mathcal{A}}_\Phi)$.
\end{proof}

\begin{prop}\label{prop: multiplicativity of ch}
For $\varnothing \neq S \subset \Phi(1)$, a non-empty set of rays,
\[
\widetilde{\ch}\left(\prod_{\rho\in S} [\mathcal{O}(\rho)]\right) = \prod_{\rho\in S}\widetilde{\ch}([\mathcal{O}(\rho)]) .
\]
\end{prop}
\begin{proof}
We proceed by induction on $|S|$. If $|S| = 1$ there is nothing to prove. From now on we assume that $|S| > 1$.

Let $\xi\in S$ and let $S' := S\smallsetminus\{\xi\}$; note that $S' \neq \varnothing$ and $|S'| < |S|$. Hence, by the induction hypothesis,
\[
\widetilde{\ch}\left(\prod_{\rho\in S'} [\mathcal{O}(\rho)]\right) = \prod_{\rho\in S'}\widetilde{\ch}([\mathcal{O}(\rho)]) .
\]
By the product formula \eqref{product formula} there exist $a_\sigma \in \mathbb{Z}$ such that
\[
\prod_{\rho\in S'} [\mathcal{O}(\rho)] = \sum_{\sigma\in\langle S'\rangle} a_\sigma [\mathcal{O}(\sigma)] .
\]
Therefore,
\[
\prod_{\rho\in S} [\mathcal{O}(\rho)] = [\mathcal{O}(\xi)]\cdot\prod_{\rho\in S'} [\mathcal{O}(\rho)] = \sum_{\sigma\in\langle S'\rangle} a_\sigma [\mathcal{O}(\xi)]\cdot[\mathcal{O}(\sigma)]
\]

We claim that 
\[
\widetilde{\ch}\left([\mathcal{O}(\xi)]\cdot[\mathcal{O}(\sigma)]\right) = \widetilde{\ch}\left([\mathcal{O}(\xi)]\right)\cdot\widetilde{\ch}\left([\mathcal{O}(\sigma)]\right) .
\]
If $\xi+\sigma\in\Phi$, the claim follows from the definition of the Chern character, so we assume that $\xi+\sigma\notin\Phi$. Let $\tau\in[\sigma]$ be a minimal cone such that $\xi+\tau\notin\Phi$. 

Note that $\tau\neq\oo$ and $\xi+\mu\in\Phi$ for any $\mu\in\partial\tau$ and it follows from Lemma \ref{lemma: mult projective space} that
\[
\widetilde{\ch}\left([\mathcal{O}(\xi)]\cdot[\mathcal{O}(\tau)]\right) = \widetilde{\ch}\left([\mathcal{O}(\xi)]\right)\cdot\widetilde{\ch}\left([\mathcal{O}(\tau)]\right) .
\]

If $\tau = \sigma$ the claim is proven. From now on we assume that $\oo\neq\tau \neq \sigma$ and proceed by induction on the dimension of $\sigma$. Let $\tau'\in[\sigma]$ denote the unique cone such that $\tau\cap\tau' = \oo$ and $\tau + \tau' = \sigma$. By the product formula \eqref{product formula} there exist $b_\mu\in\mathbb{Z}$ such that
\[
[\mathcal{O}(\xi)]\cdot[\mathcal{O}(\sigma)] = [\mathcal{O}(\xi)]\cdot[\mathcal{O}(\tau)]\cdot[\mathcal{O}(\tau')] = \sum_{\mu\in\langle\{\xi\}\cup[\tau](1)\rangle} b_\mu\cdot[\mathcal{O}(\mu)]\cdot[\mathcal{O}(\tau')] .
\]
Since $[\mathcal{O}(\mu)]\cdot[\mathcal{O}(\tau')] = \prod_{\rho\in[\mu](1)\cup[\tau'](1)}[\mathcal{O}(\rho)]$ and $|\rho\in[\mu](1)\cup[\tau'](1)| < |S|$ the induction hypothesis implies that
\begin{eqnarray*}
\widetilde{\ch}\left([\mathcal{O}(\mu)]\cdot[\mathcal{O}(\tau')]\right) & = & \widetilde{\ch}\left([\mathcal{O}(\mu)]\right)\cdot\widetilde{\ch}\left([\mathcal{O}(\tau')]\right) , \\
\widetilde{\ch}\left([\mathcal{O}(\xi)]\cdot[\mathcal{O}(\tau)]\right) & = & \widetilde{\ch}\left([\mathcal{O}(\xi)]\right)\widetilde{\ch}\left([\mathcal{O}(\tau)]\right) .
\end{eqnarray*}
Therefore,
\begin{eqnarray*}
\widetilde{\ch}\left([\mathcal{O}(\xi)]\cdot[\mathcal{O}(\sigma)]\right) & = & \sum_{\mu\in\langle\{\xi\}\cup[\tau](1)\rangle} b_\mu\cdot\widetilde{\ch}\left([\mathcal{O}(\mu)]\cdot[\mathcal{O}(\tau')]\right) \\
& = & \sum_{\mu\in\langle\{\xi\}\cup[\tau](1)\rangle} b_\mu\cdot\widetilde{\ch}\left([\mathcal{O}(\mu)]\right)\cdot\widetilde{\ch}\left([\mathcal{O}(\tau')]\right) \\
& = & \widetilde{\ch}\left(\sum_{\mu\in\langle\{\xi\}\cup[\tau](1)\rangle} b_\mu\cdot[\mathcal{O}(\mu)]\right)\cdot\widetilde{\ch}\left([\mathcal{O}(\tau')]\right) \\
& = & \widetilde{\ch}\left([\mathcal{O}(\xi)]\cdot[\mathcal{O}(\tau)]\right)\cdot\widetilde{\ch}\left([\mathcal{O}(\tau')]\right) \\
& = & \widetilde{\ch}\left([\mathcal{O}(\xi)]\right)\widetilde{\ch}\left([\mathcal{O}(\tau)]\right)\cdot\widetilde{\ch}\left([\mathcal{O}(\tau')]\right) \\
& = & \widetilde{\ch}\left([\mathcal{O}(\xi)]\right)\cdot\widetilde{\ch}\left([\mathcal{O}(\sigma)]\right)
\end{eqnarray*}

If follows that
\begin{eqnarray*}
\widetilde{\ch}\left(\prod_{\rho\in S} [\mathcal{O}(\rho)]\right) & \stackrel{\eqref{product formula}}{=} & \sum_{\sigma\in\langle S'\rangle} a_\sigma\cdot \widetilde{\ch}\left([\mathcal{O}(\xi)]\cdot[\mathcal{O}(\sigma)]\right) \\ 
& = & \sum_{\sigma\in\langle S'\rangle} a_\sigma\cdot\widetilde{\ch}\left([\mathcal{O}(\xi)]\right)\cdot\widetilde{\ch}\left([\mathcal{O}(\sigma)]\right) \\
& = & \widetilde{\ch}\left([\mathcal{O}(\xi)]\right)\cdot \sum_{\sigma\in\langle S'\rangle} a_\sigma\cdot\widetilde{\ch}\left([\mathcal{O}(\sigma)]\right) \\
& = & \widetilde{\ch}\left([\mathcal{O}(\xi)]\right)\cdot \widetilde{\ch}\left(\prod_{\rho\in S'} [\mathcal{O}(\rho)]\right) \\
\text{by the inductive hypothesis} & = & \widetilde{\ch}\left([\mathcal{O}(\xi)]\right)\cdot \prod_{\rho\in S'}\widetilde{\ch}([\mathcal{O}(\rho)]) \\
& = & \prod_{\rho\in S}\widetilde{\ch}([\mathcal{O}(\rho)])
\end{eqnarray*}
\end{proof}

\subsection{The Todd class}
Let
\[
\widetilde{\Td}(\Phi) := \prod_{\rho\in\Phi(1)} \dfrac{\phi_\rho}{1-\exp(-\phi_\rho)} \in H^0(\Phi;\widehat{\mathcal{A}}_\Phi)
\]
Let $\Td(\Phi)\in H(\Phi)$ denote the image of $\widetilde{\Td}(\Phi)$ under the map $H^0(\Phi;\widehat{\mathcal{A}}_\Phi) \to H(\Phi)$.

\begin{thm}[\cite{I}, Theorem 3.3]\label{thm: genus one}
Suppose that $\Phi$ is a complete unimodular simplicial fan. Then,
\[
\int_\Phi \Td(\Phi) = 1 .
\]
\end{thm}

\subsection{The Hirzebruch's $L$-class}
Let
\[
\widetilde{L}(\Phi) = \widetilde{\ch}(\sum_i[\Omega^i_\Phi])\widetilde{\Td}(\Phi)
\]
Let $L(\Phi)\in H(\Phi)$ denote the image of $\widetilde{L}(\Phi)$ under the map $H^0(\Phi;\widehat{\mathcal{A}}_\Phi) \to H(\Phi)$.

The exact sequence \eqref{ses forms big} shows that
\[
[\Omega^1_\Phi]+[\left(H^1\Omega^1_\Phi\right)_\Phi] = \sum_{\rho\in\Phi(1)}[\mathcal{O}(\rho)]\]
in $K(\Phi)$. Therefore,
\[
\sum_i[\Omega^i_\Phi]\ \bullet\ \sum_j [\EP^jH^1(\Phi;\Omega^1_\Phi)_\Phi] = \sum_k\EP^k(\bigoplus_{\rho\in\Phi(1)}[\mathcal{O}(\rho)]) = \prod_{\rho\in\Phi(1)} ([\mathcal{O}(\rho)]+1) ,
\]
Let $h_1(\Phi) := \dim H^2(\Phi)$. Thus, 
\[
2^{h_1(\Phi)}\sum_i[\Omega^i_\Phi] = \prod_{\rho\in\Phi(1)} ([\mathcal{O}(\rho)]+1) .
\]
It follows from Proposition \ref{prop: multiplicativity of ch} that
\[
\widetilde{\ch}(\prod_{\rho\in\Phi(1)} ([\mathcal{O}(\rho)]+1)) = \prod_{\rho\in\Phi(1)} (\widetilde{\ch}([\mathcal{O}(\rho)]) + 1).
\]
Therefore,
\[
\widetilde{\ch}(\sum_i[\Omega^i_\Phi]) = \dfrac1{2^{h_1(\Phi)}}\prod_{\rho\in\Phi(1)} (\widetilde{\ch}([\mathcal{O}(\rho)]) + 1).
\]

It follows that
\begin{multline*}
L(\Phi) = 2^{2n}\prod_{\rho\in\Phi(1)}\frac{1+\exp(-\phi_\rho)}{1-\exp(-\phi_\rho)} \frac{ \phi_\rho }{2}
    = 2^{2n}\prod_{\rho\in\Phi(1)}\frac{1+\exp\left(-2\dfrac{ \phi_\rho }{2}\right)}{1-\exp\left(-2\dfrac{ \phi_\rho }{2}\right)} \dfrac{ \phi_\rho }{2}\\
    =2^{2n}\prod_{\rho\in\Phi(1)}\frac{\dfrac{ \phi_\rho }{2}}{\tanh\left(\dfrac{ \phi_\rho }{2}\right)}
    =2^{2n}\prod_{\rho\in\Phi(1)} \left(1-\sum_{k\geq 1}(-1)^k\frac{B_{2k}}{(2k)!} \phi_\rho ^{2k} \right),
\end{multline*}
where $B_{2k}$ is the absolute value of the $2k$-th Bernoulli number. Therefore the top degree term of $L(\Phi)$ satisfies
\begin{equation}\label{the L polynomial}
(-1)^nL_{2n}(\Phi) = 2^{2n}\sum_{k=1}^n \sum_{\substack{m_1+\dots+m_k = n,\\
    m_i>0,\\
    \rho_1,\dots,\rho_k\in\Phi(1),\\
    \rho_{i}\neq \rho_j}}(-1)^k \frac{B_{2m_1}}{(2m_1)!}\cdot\ldots\cdot\frac{B_{2m_k}}{(2m_k)!}\phi_{\rho_1}^{2m_1}\cdot\ldots\cdot \phi_{\rho_k}^{2m_k} .
\end{equation}

\subsection{Riemann-Roch type theorem}
The proof of Theorem \ref{thm: RR} below is an adaptation to the present context of the proof  due to H.~Schenck (\cite{S}) of the Hirzebruch-Riemann-Roch Theorem for toric varieties.

\begin{thm}\label{thm: RR}
Suppose that $\Phi$ is a complete unimodular simplicial fan. Then, for $\gamma\in K(\Phi)$
\begin{equation}\label{RR}
\chi(\gamma) = \int_\Phi \ch(\gamma)\cdot\Td(\Phi) .
\end{equation}
\end{thm}
\begin{proof}
By Proposition \ref{prop: bases for K} it is sufficient to show that \eqref{RR} holds for
$\gamma$ equal to $[\mathcal{O}]$, $[\mathcal{O}(\sigma)]$, $\oo\neq\sigma\in\Phi$. Since $\chi([\mathcal{O}]) = 1$ and $\ch([\mathcal{O}]) = 1$ the equality \eqref{RR} with $\gamma=[\mathcal{O}]$ follows from Theorem \ref{thm: genus one}.

For $\oo\neq\sigma\in\Phi$ it follows from Lemma \ref{lemma: line bundles vs injectives} that
\[
\chi([\mathcal{O}(\sigma)]) = \sum_{\tau\leqslant\sigma}(-1)^{d(\tau)}\chi(\mathbb{R}_{\Star(\tau)}) = 0.
\]
and therefore it is sufficient to show that
\begin{equation}\label{RR vanishing}
\int_\Phi \ch([\mathcal{O}(\sigma)])\cdot\Td(\Phi) = 0.
\end{equation}
The proof of \eqref{RR vanishing} proceeds by induction on the dimension of $\Phi$.

The case $\dim\Phi = 1$, i.e. $\Phi = \mathbb{P}^1$, follows from direct calculation left to the reader. We now assume that \eqref{RR vanishing} holds for all complete unimodular fans of dimension smaller than $\dim\Phi$.

It follows from Theorem \ref{thm: genus one} that for $\rho\in\Phi(1)$ the equality \eqref{RR vanishing} is equivalent to
\[
\int_\Phi (1-\exp(- \phi_\rho ))\cdot\Td(\Phi) = 1 .
\]
Since $\Td(\Phi) = \prod_{\xi\in\Phi(1)}\dfrac{ \phi_\xi}{1-\exp(- \phi_\xi)}$ and $ \phi_\rho  \cdot\phi_\xi=0$ for $\xi\notin\bdstar_\Phi(\rho)(1)\cup \{\rho\}$ it follows that
\[
(1-\exp(- \phi_\rho ))\Td(\Phi) =  \phi_\rho\prod_{\xi\in\Phi(1)\smallsetminus\{\rho\}}\dfrac{ \phi_\xi}{1-\exp(- \phi_\xi)} = 
 \phi_\rho\prod_{\xi\in\bdstar_\Phi(\rho)(1)}\dfrac{ \phi_\xi}{1-\exp(- \phi_\xi)}
\]
By Theorem \ref{thm: genus one}, Proposition \ref{prop: gysinstars} and induction on dimension of the fan
\[
\int\limits_\Phi (1-\exp(- \phi_\rho ))\cdot\Td(\Phi) = \int\limits_\Phi  \phi_\rho\prod_{\xi\in\bdstar_\Phi(\rho)(1)}\dfrac{ \phi_\xi}{1-\exp(- \phi_\xi)}
= \int\limits_{\overline{\Star}_\Phi(\rho)} \Td(\overline{\Star}_\Phi(\rho)) = 1
\]
For general cones we shall prove \eqref{RR vanishing} by induction on dimension of the cone, the base of the induction being the case of rays above. Thus, let $\sigma\in\Phi$, $\dim(\sigma) > 1$ and assume that
\begin{equation}\label{inductive RR vanishing}
\int_\Phi \ch([\mathcal{O}(\tau)])\cdot\Td(\Phi) = 0
\end{equation}
for all cones $\tau\in\Phi$ with $\dim(\tau) < \dim(\sigma)$.

Let $\rho\in[\sigma](1)$ with the opposite facet $\tau$ so that $\sigma = \rho + \tau$. Then, by definition, $\mathcal{O}(\sigma) = \mathcal{O}(\rho)\otimes\mathcal{O}(\tau)$ and $\ch([\mathcal{O}(\sigma)]) = \ch([\mathcal{O}(\rho)])\cdot\ch([\mathcal{O}(\tau)])$. In view of \eqref{inductive RR vanishing} it is sufficient to show that
\[
\int_\Phi \ch([\mathcal{O}(\tau)])\cdot(1-\ch([\mathcal{O}(\rho)])) \cdot\Td(\Phi) = 0
\]
As before, by Proposition \ref{prop: gysinstars} and induction on dimension of the fan
\begin{multline*}
\int\limits_\Phi \ch([\mathcal{O}(\tau)])\cdot(1-\exp(- \phi_\rho ))\cdot\Td(\Phi) \\
= \int\limits_\Phi \ch([\mathcal{O}(\tau)])\cdot \phi_\rho\prod_{\xi\in\bdstar_\Phi(\rho)(1)}\dfrac{ \phi_\xi}{1-\exp(- \phi_\xi)} \\
= \int\limits_{\overline{\Star}_\Phi(\rho)}\!\!\!\!\! \ch([\mathcal{O}(\tau)])\cdot\Td(\overline{\Star}_\Phi(\rho)) = 0
\end{multline*}
\end{proof}

\subsection{Hirzebruch signature theorem for unimodular simplicial fans}

\begin{thm}\label{thm: signature}
Suppose that $\Phi$ is a complete unimodular fan of even dimension. Then,
\[
\sign(\Phi) = \int_\Phi L(\Phi) .
\]
\end{thm}
\begin{proof}
By Theorem \ref{thm: signature is euler characteristic} and Proposition \ref{prop: coh of forms}
\[
\sign(\Phi) = \sum_i (-1)^i\dim H^i(\Phi;\Omega^i_\Phi) = \sum_i \chi(\Phi;\Omega^i_\Phi) .
\]
It follows from Theorem \ref{thm: RR} that
\[
\sign(\Phi) = \int_\Phi L(\Phi) .
\]
\end{proof}

\section{Application: signature of locally convex fans}\label{section: locally convex fans}
Suppose that $\Phi$ is a fan.

\begin{definition}
A fan $\Phi$ is called \emph{locally convex at $\sigma\in\Phi$} if the subset $|\Star_\Phi(\sigma)|$ is convex.

A fan $\Phi$ is called \emph{locally convex} if it is locally convex at every cone.
\end{definition}

\begin{lemma}
Suppose that $\Phi$ is locally convex at $\sigma\in\Phi$. Then, the fan $\overline{\Star}_\Phi(\sigma)$ is locally convex.
\end{lemma}

\begin{thm}\label{thm: sign of signature}
Suppose that $\Phi$ is a locally convex complete unimodular fan of dimension $\dim\Phi = 2n$. Then,
\[
(-1)^n\sign(\Phi) \geqslant 0 .
\] 
\end{thm}
\begin{proof}
In view of Theorem \ref{thm: signature} and the formula \eqref{the L polynomial} it suffices to show that for any collection of distinct $k$ rays $\rho_1,\dots,\rho_k\in\Phi(1)$ and integers $m_1,\dots,m_k$ such that $m_i \geqslant 1$ and $m_1+\dots m_k=n$ the inequality 
\[
\int_\Phi(-1)^k \phi_{\rho_1}^{2m_1}\cdot\hdots\cdot \phi_{\rho_k}^{2m_k}\geqslant 0
\]
holds.

If $\rho_1+\cdots +\rho_k$ is not a cone in $\Phi$ the intersection of the supports of the functions $\phi_{\rho_i}$ is empty, hence the integral vanishes.

Suppose that $\rho_1+\cdots +\rho_k = \tau \in \Phi$. Since $\Phi$ is locally convex, each function $-\phi_{\rho_i}$ is linearly equivalent to a function $\ell_i$ which, when restricted to $\overline{\Star}_\Phi(\rho_i)$, and, hence, to $\overline{\Star}_\Phi(\tau)$, is strictly convex. Therefore
\begin{eqnarray*}
\int_\Phi(-1)^k \phi_{\rho_1}^{2m_1}\cdot\hdots\cdot \phi_{\rho_k}^{2m_k} & = & \int_\Phi \left(\phi_{\rho_1}\cdot\hdots\cdot \phi_{\rho_k}\right)\cdot (-\phi_{\rho_1})^{2m_1-1}\cdot\hdots\cdot (-\phi_{\rho_k})^{2m_k-1} \\
    & = & \int_{\overline{\Star}_\Phi(\tau)} \ell_1^{2m_1-1} \cdot \hdots\cdot \ell_k^{2m_k-1} \geqslant 0
\end{eqnarray*}
by Proposition \ref{prop: nonnegative convex monomials}.
\end{proof}


\begin{thebibliography}{ABCD}
\bibitem[B]{B} M.~Brion, The structure of the polytope algebra, \emph{T\^ohoku Math. J.} \textbf{49} (1997), 1 -- 32.

\bibitem[BBFK]{BBFK} G.~Barthel, J.-P.~Brasselet, K.-H.~Fieseler, L.~Kaup,
Combinatorial intersection homology for fans. in \emph{Tohoku Math. J. (2)} \textbf{54} (2002), 1 -- 41.

\bibitem[BrL1]{HLIJ} P.~Bressler, V.~Lunts, Hard Lefschetz theorem and Hodge-Riemann relations for intersection cohomology of nonrational polytopes, \emph{Indiana Univ. Math. J.} \textbf{54} (2005), 263 -- 307.

\bibitem[BrL2]{ICNP} P.~Bressler, V.~Lunts, Intersection cohomology on nonrational polytopes, \emph{Compositio Matematica} \textbf{135} (2003), 245 -- 278.

\bibitem[I]{I} M.~Ishida, Polyhedral Laurent series and Brion’s equalities, \emph{Intl. J. Math.} \textbf{1} (1990), 251 -- 265.

\bibitem[K]{K} A.~Kirillov~Jr., Quiver representations and quiver varieties. \emph{American Mathematical Society (AMS)} (2016)
\bibitem[LR]{LR} N.C.~Leung, V.~Reiner, The signature of a toric variety, \emph{Duke Mathematical Journal} \textbf{111} (2002), 253 -- 286.

\bibitem[S]{S} H.~Schenck, Toric Hirzebruch-Riemann-Roch via Ishida’s theorem on the Todd genus,
\emph{Proc. Amer. Math. Soc.} \textbf{141} (2013), 1215--1217

\bibitem[W]{W} J.~W\l{}odarczyk, Decomposition of Birational Toric Maps in Blow-Ups and Blow-Downs, \emph{Transactions of the AMS} \textbf{349} no. 1 (1997), 373--411.
\end{thebibliography}
\end{document}